\documentclass[11pt]{amsart}

\usepackage{enumitem, amsmath, amsthm, amsfonts, amssymb, xy,  mathrsfs, graphicx, lscape, mathdots}
\usepackage[usenames, dvipsnames]{color}
\usepackage[margin=1in]{geometry} 
\usepackage{tikz}
\usetikzlibrary{matrix,arrows}
\usepackage{tabularx, array}

\usepackage{tikz-cd}
\usepackage{xcolor}
\usepackage{float}
\usepackage{stackengine}

\numberwithin{equation}{section}
\newtheorem{theorem}[equation]{Theorem}

\newtheorem{proposition}[equation]{Proposition}
\newtheorem{lemma}[equation]{Lemma}

\newtheorem{notation}[equation]{Notation}

\theoremstyle{definition}
\newtheorem{rmk}[equation]{Remark}
\newenvironment{remark}[1][]{\begin{rmk}[#1] \pushQED{\qed}}{\popQED \end{rmk}}
\newtheorem{eg}[equation]{Example}

\newtheorem{defn}[equation]{Definition}
\newenvironment{definition}[1][]{\begin{defn}[#1]\pushQED{\qed}}{\popQED \end{defn}}
\newtheorem{ques}[equation]{Question}

\usepackage[in]{fullpage}
\usepackage[colorlinks=true, pdfstartview=FitV, linkcolor=black, citecolor=black, urlcolor=black]{hyperref}

\newcommand{\ZZ}{\mathbb{Z}}

\newcommand{\RR}{\mathbb{R}}
\newcommand{\ddim}{\underline{\dim}}

\renewcommand{\phi}{\varphi}

\newcommand{\xto}[1]{\xrightarrow{#1}}

\DeclareMathOperator{\rep}{rep}
\DeclareMathOperator{\Hom}{Hom}

\newcommand{\comment}[1]{}

\makeatletter
\def\Ddots{\mathinner{\mkern1mu\raise\p@
\vbox{\kern7\p@\hbox{.}}\mkern2mu
\raise4\p@\hbox{.}\mkern2mu\raise7\p@\hbox{.}\mkern1mu}}
\makeatother

\usetikzlibrary{decorations.pathmorphing}

\DeclareMathOperator{\coker}{coker}
\DeclareMathOperator{\op}{op}
\DeclareMathOperator{\rad}{rad}
\DeclareMathOperator{\AR}{AR}

\begin{document}
\title{Total stability and Auslander-Reiten theory for Dynkin quivers}

\author{Yariana Diaz}
\address{Macalester College, Department of Mathematics, Statistics, and Computer Science, St. Paul, MN, USA}
\email[Yariana Diaz]{ydiaz@macalester.edu}

\author{Cody Gilbert}
\address{Saint Louis University, Department of Mathematics, St. Louis, USA}
\email[Cody Gilbert]{cody.gilbert@slu.edu}

\author{Ryan Kinser}
\address{University of Iowa, Department of Mathematics, Iowa City, USA}
\email[Ryan Kinser]{ryan-kinser@uiowa.edu}
\thanks{This work was supported by a grant from the Simons Foundation (636534, RK).  This material is based upon work supported by the National Science Foundation under Award No. DMS-2303334.}

\begin{abstract}
This paper concerns stability functions for Dynkin quivers, in the generality introduced by Rudakov.
We show that relatively few inequalities need to be satisfied for a stability function to be totally stable (i.e. to make every indecomposable stable).
Namely, a stability function $\mu$ is totally stable if and only if $\mu(\tau V) < \mu(V)$ for every almost split sequence $0 \to \tau V \to E \to V \to 0$ where $E$ is indecomposable.  These can be visualized as those sequences around the ``border'' of the Auslander-Reiten quiver.
\end{abstract}

\makeatletter
\@namedef{subjclassname@2020}{%
  \textup{2020} Mathematics Subject Classification}
\makeatother



\maketitle

\setcounter{tocdepth}{1}
\tableofcontents

\section{Introduction}\label{sec:intro}
\subsection{Motivation and context}
This paper concerns stability for quivers in the generality defined by Rudakov \cite{Rudakov97}, following the formalism of stability functions presented in \cite{BST} (see Definition \ref{def:stability}).  These are functions from the collection of isomorphism classes of nonzero representations to a totally ordered set which satisfy the see-saw property \eqref{eq:seesaw}. A quiver representation $V$ is stable with respect to a stability function $\mu$ when $\mu(W) < \mu(V)$ for all proper, nonzero subrepresentations $W$ of $V$.  This is more general than Bridgeland's stability conditions in \cite{Bridgeland07}, also called linear stability conditions or central charges.
Other notions of stability for quiver representations were also introduced in \cite{Schofield91, King94}.
Stability conditions have found connections across various areas of mathematics such as moduli spaces of representations, semi-invariants, Harder-Narasimhan filtrations, quantum dilogarithm identities, and green sequences and paths.

We are specifically interested in stability functions which make every indecomposable representation of a quiver stable.
Such a stability function is called \emph{totally stable}, and seems to have first been examined (for a special type of linear stability condition) by Reineke \cite{Reineke03}.  Interest in total stability for linear stability conditions has picked up in recent years \cite{Qiu15,Igusa20,HH,AI,Kinser22,Marczinzik,tattar2023weak,CQZ,qiu2025geometric}, particularly as connections with other topics have emerged, including
Coxeter combinatorics \cite{BGMS,defant2025permutahedron}
algebraic geometry \cite{KOT21,Qiu-gldim,Qiu-contractible,hu2025moduli}, and 
persistence theory \cite{fersztand2024harder}.

In considering total stability, we can immediately restrict our attention to Dynkin quivers, as every endomorphism of a stable representation is a scalar multiple of the identity, while every non-Dynkin quiver has indecomposable representations with nilpotent endomorphisms, and thus cannot be stable (see \cite{DR76} or \cite[Ch.~7]{DWbook}).
For Dynkin quivers, the structure of the module category is very combinatorial and nicely encoded by the Auslander-Reiten quiver; this is the main tool we use throughout this paper.
We work over an arbitrary base field which does not require its own notation.

\subsection{Results}
An elementary reduction shows that $\mu$ is totally stable if and only if $\mu(W)<\mu(V)$ for all pairs of nonzero \emph{indecomposable} representations $W<V$ (see Lemma \ref{lem:indecomp}).
The third column of the table below shows the number of pairs of distinct indecomposables for each Dynkin type, a rough estimate of the number of inequalities which would need to be considered with only this reduction.  

Our main theorem gives a criteria for a stability function to be totally stable that significantly reduces the number of pairs which need to be considered.  
The fourth column of the table below shows the much smaller number of inequalities which need to be considered according to Theorem \ref{thm:main}.

\begin{theorem}\label{thm:main}
Let $Q$ be a Dynkin quiver and $\mu\colon \rep^*(Q)\to \mathcal{P}$ be a stability function on $Q$, where $\mathcal{P}$ is a totally ordered set.  
We have that $\mu$ is totally stable if and only if
$\mu(\tau V) < \mu(V)$
for all almost-split sequences $0 \to \tau V \to E \to V \to 0$ such that $E$ is indecomposable (where $\tau$ is the Auslander-Reiten translation).
\end{theorem}

\begin{center}
{\renewcommand{\arraystretch}{1.25}
\begin{tabularx}{0.9\textwidth} { 
  | >{\centering\arraybackslash}X  
  | >{\centering\arraybackslash}X 
  | >{\centering\arraybackslash}X  
  | >{\centering\arraybackslash}X | }
 \hline
 Type & Number of Indecomposables & Number of Stability Inequalities & Number of Inequalities using Theorem \\
 \hline\hline
 $\mathbb{A}_n$ & $\frac{n(n+1)}{2}$ & $o(n^4)$ & $n-1$ \\
 \hline
 $\mathbb{D}_n$ & $n(n-1)$ & $o(n^4)$ & $3(n-2)$ \\
 \hline
 $\mathbb{E}_6$ & 36 & 1260 & 15 \\
 \hline
 $\mathbb{E}_7$ & 63 & 3906 & 24 \\
 \hline
 $\mathbb{E}_8$ & 120 & 14280 & 42 \\ 
 \hline
\end{tabularx}}
\end{center}

The forward direction of the theorem follows from definitions, so we only need to prove that the smaller number of inequalities suffices to verify total stability.
The proof uses combinatorics of Auslander-Reiten quivers (reviewed in \S\ref{sec:AR}), proceeding in cases by Dynkin type.  Type $\mathbb{A}$ was already essentially completed in \cite{Kinser22} but in different language, so we include this case for completeness.
Of course, the combinatorics of Auslander-Reiten quivers is more complicated in types $\mathbb{D}$ and $\mathbb{E}$, but fortunately the general strategy of our type $\mathbb{D}$ proof can be intuitively extended to type $\mathbb{E}$.

\begin{remark}
This project began with a computational investigation of \cite[Conjecture~7.1]{Reineke03}, which considers standard linear stability functions, that is, those of the form $\mu(V) = (\theta\cdot \ddim V)/\dim V$, where $\theta \in \ZZ^{Q_0}$ is an integral weight, $\cdot$ is the usual dot product, and $\dim V$ is the total vector space dimension of $V$.  We found that the conjecture is:
\begin{itemize}
    \item true for all orientations in type $\mathbb{A}_n$ for all $n$ (already proven in citations above);
    \item true in type $\mathbb{D}_n$ for $n \leq 8$; 
    \item false for certain orientations in type $\mathbb{D}_n$ for all $n \geq 9$;
    \item true for all orientations in type $\mathbb{E}_6$;
    \item false for certain orientations in types $\mathbb{E}_7$ and $\mathbb{E}_8$.
\end{itemize}
Independently, Marczinzik recently published an easily verifiable counterexample \cite{Marczinzik}.  Furthermore, in references cited above, a modification of the conjecture using more general linear stability conditions is proposed. The modified version replaces $\dim V$ in the denominator with $w\cdot \dim V$, where $w \in \mathbb{Z}_{>0}$ is a positive weight, and was proven in \cite{CQZ}.
\end{remark}

\subsection*{Acknowledgements}
We thank the QPA team for development of the software \cite{QPA} which was used for many explorations that led to discovering the results of this paper. We also thank the referee for suggestions leading to expository improvements.

\section{Background and notation}\label{sec:background}
\subsection{Quiver representations}
We recall just enough here to establish our notation; an introduction to quiver representations can be found in texts such as \cite{assemetal,Schiffler:2014aa,DWbook}.
We write $Q_0$ for the set of vertices of a quiver $Q$, and $Q_1$ for its set of arrows, while $sa$ and $ta$ denote the \emph{source} and \emph{target} of an arrow $sa \xrightarrow{a} ta$.
The category of finite dimensional representations of $Q$ is written $\rep(Q)$, and the collection of nonzero objects of $\rep(Q)$ is denoted by $\rep^*(Q)$.
The \emph{dimension vector} of a representation $V$ of $Q$ is denoted by $\ddim V \in \RR^{Q_0}$.
We denote by $P_x$ the indecomposable projective representation associated to a vertex $x$.
Similarly, $I_x$ denotes the indecomposable injective representation associated to vertex $x$.

\subsection{Stability}\label{sec:stability}
We use the definition of general algebraic stability introduced by Rudakov \cite{Rudakov97}, following the exposition of \cite{BST}.

\begin{definition}\label{def:stability}
Let $(\mathcal{P}, \leq)$ be a totally ordered set.  A function $\mu\colon \rep^*(Q) \to \mathcal{P}$ which is constant on isomorphism classes is a \emph{stability function for $Q$} if, for each short exact sequence of nonzero objects
$0\to A \to B \to C \to 0$ in $\rep(Q)$, exactly one of the following holds:

\begin{equation}\label{eq:seesaw}
    \begin{split}
    \text{either}\quad & \mu(A) < \mu(B) < \mu(C),\\
    \text{or}\quad & \mu(A) > \mu(B) > \mu(C),\\
    \text{or}\quad & \mu(A) = \mu(B) = \mu(C).
    \end{split}
\end{equation}
The condition immediately above is called the \emph{see-saw property} for $\mu$ and $\mu(X)$ is the \emph{phase} or \emph{slope} of $X$.
\end{definition}

A nonzero representation $V$ of $Q$ is called $\mu$-\emph{stable} if $\mu(W) < \mu(V)$ for all nonzero, proper subrepresentations $0< W<V$.
This work investigates stability conditions $\mu$ such that every indecomposable representation of $Q$ is $\mu$-stable.  This immediately restricts our attention to $Q$ of Dynkin type since the endomorphism ring of a stable representation must be a division ring.

\begin{definition}\label{def:weights}
A stability condition $\mu$ for a quiver $Q$ is \emph{totally stable}, or a \emph{total stability condition}, if every indecomposable representation of $Q$ is $\mu$-stable.
\end{definition}

The following elementary lemma gives a useful first reduction for checking a stability condition is totally stable.

\begin{lemma}\label{lem:indecomp}
Let $Q$ be an arbitrary quiver and $\mu$ a stability function for $Q$.  A representation $V$ of $Q$ is $\mu$-stable if and only if $\mu(W) < \mu(V)$ for all proper nonzero indecomposable subrepresentations $W < V$.
\end{lemma}
\begin{proof}
The forward implication is by definition, so let us assume for the converse that $\mu(W) < \mu(V)$ for all proper nonzero indecomposable $W<V$.
We first note that $\mu(A\oplus B)$ lies between $\mu(A)$ and $\mu(B)$ for any nonzero $A, B$, by considering the see-saw property for a split short exact sequence.

With this, let $Y < V$ be an arbitrary proper nonzero subrepresentation.  From repeated application of the observation above, $Y$ has an indecomposable summand $X$ such that $\mu(X) \geq \mu(Y)$.  Since $\mu(X) < \mu(V)$ by assumption, this shows $\mu(Y)<\mu(V)$ as well.
\end{proof}

\subsection{Auslander-Reiten quivers}\label{sec:AR}
Given a quiver $Q$, the Auslander-Reiten quiver of $Q$ is a useful tool for visualizing the structure of the category $\rep(Q)$.  When $Q$ is Dynkin, it encodes essentially all of the information about the category $\rep(Q)$ in combinatorics.
The reader may wish to consult standard references such as \cite{assemetal, Schiffler:2014aa, DWbook} for more detail on Ausland-Reiten theory.

The \emph{Auslander-Reiten quiver} of a quiver $Q$, denoted $\AR(Q)$, is
the quiver whose vertices correspond to the isomorphism classes of indecomposable representations of $Q$,
and arrows correspond to irreducible morphisms,
with $\dim{\rad(V,W)/\rad^2(V,W)}$ arrows from $V$ to $W$. 
We often use that the duality functor on representations induces an identification between $\text{AR}(Q^{\text{op}})$ and $\text{AR}(Q)^{\text{op}}$.

Let $\tau$ be the Auslander-Reiten translation and $\tau^{-1}$ the inverse Auslander-Reiten translation. The \emph{$\tau$-orbit of $V\in \rep(Q)$} is the set of nonzero translates $\tau^k V$, for $k \in \ZZ$. Standard facts are that $\tau$ maps projectives to 0 and other indecomposables to indecomposables, while $\tau^{-1}$ maps injectives to 0 and other indecomposables to indecomposables.

Exact sequences of the form
\[
0 \rightarrow L \xrightarrow[]{f} M \xrightarrow[]{g} N \rightarrow 0
\] 
with $L, N$ indecomposable and $f,g$ irreducible are almost split and the converse is also true (see \cite[Theorem 5.4.8]{DWbook}, for example).
Additionally, when $N$ is indecomposable and not projective, almost split sequences are of the form 
\begin{equation}
0\rightarrow \tau N\rightarrow M\rightarrow N\rightarrow 0,
\end{equation} 
and conversely, every sequence of the above form is almost split.
It follows that the Auslander-Reiten quiver summarizes information about the indecomposable modules, irreducible morphisms, and almost split sequences in the category $\rep(Q)$.

Irreducible morphisms are either monomorphisms or epimorphisms \cite[Lemma 5.2.2]{DWbook}; furthermore, if $f\colon L\rightarrow M$ is an irreducible
monomorphism, then $\coker f$ is
indecomposable \cite[Corollary 5.2.7]{DWbook}.
Similarly, if $g\colon M \rightarrow N$ is an
irreducible epimorphism, then $\ker g$ is indecomposable. Thus for each arrow of $\text{AR}(Q)$, the kernel or cokernel (whichever is nonzero) uniquely determines a third vertex of $\text{AR}(Q)$ and corresponding short exact sequence of indecomposables. This observation is essential throughout this paper.

We also note that for Dynkin type quivers, an indecomposable representation is uniquely determined (up to isomorphism) by its dimension vector, and the set of dimension vectors of indecomposable representations only depends on the underlying graph, not the orientation of the arrows.  The orientation arrows in $Q$ does affect the orientation of the arrows in $\text{AR}(Q)$, and one fact we use about this is that $Q^{\textrm{op}}$ appears as a subquiver of $\text{AR}(Q)$ where the vertices of $Q^{\textrm{op}}$ correspond to indecomposable projectives in $\text{AR}(Q)$.

For Dynkin quivers, almost split sequences are represented in $\text{AR}(Q)$ by three different kinds of meshes \cite[\S 1.5]{Schiffler:2014aa}
\[\begin{tikzcd}
	& {U_1} &&& {U_1} &&& {U_1} \\
	{\tau V} && V & {\tau V} && V & {\tau V} & {U_2} & V \\
	&&&& {U_2} &&& {U_3}
	\arrow[from=2-1, to=1-2]
	\arrow[from=1-2, to=2-3]
	\arrow[from=2-4, to=1-5]
	\arrow[from=2-4, to=3-5]
	\arrow[from=3-5, to=2-6]
	\arrow[from=1-5, to=2-6]
	\arrow[from=2-7, to=1-8]
	\arrow[from=2-7, to=2-8]
	\arrow[from=2-7, to=3-8]
	\arrow[from=3-8, to=2-9]
	\arrow[from=2-8, to=2-9]
	\arrow[from=1-8, to=2-9]
\end{tikzcd}\]
with the corresponding almost split sequences having the form 
\[
0\rightarrow \tau V\rightarrow \bigoplus_{i} U_i\rightarrow V\rightarrow 0.
\]

The first two types of meshes are found in all Dynkin types; however, only
Dynkin quivers of type $\mathbb{D}$ and $\mathbb{E}$ have meshes of the third type, which will be referred
to as \emph{triple meshes}. Auslander-Reiten quivers for type $\mathbb{D}$ and $\mathbb{E}$ quivers will always have just one triple mesh, up to Auslander-Reiten translates, due to the unique branch point found in both quiver types. Within the middle of the triple mesh, there are two $\tau$-orbits: $U_2$ and $V$ belong to different $\tau$-orbits in the triple mesh above. 

Recall that a subquiver $Q'$ of a quiver $Q$ is said to be \emph{convex} in $Q$ if, for any path $x_0\rightarrow x_1\rightarrow\ldots\rightarrow x_l$ in $Q$ with $x_0, x_l\in Q'_0$, we have $x_i\in Q'_0$ for all $0\leq i\leq l$.
A convex region of an Auslander-Reiten quiver of the following form will be called a \emph{ladder} in $\text{AR}(Q)$ and the arrow labeled $a$ the \emph{top rung} of the ladder.
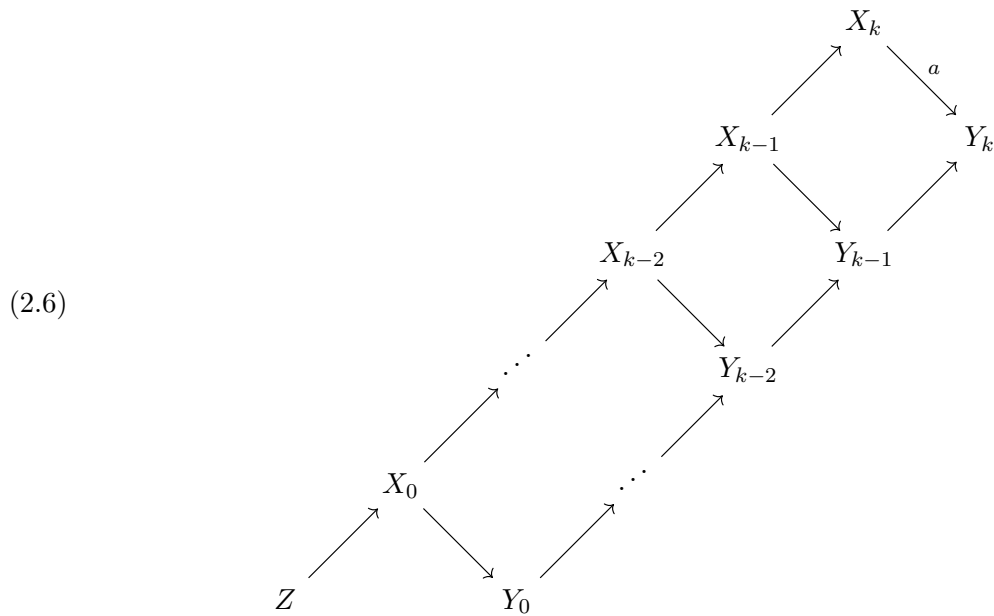
\begin{equation}\label{eq:ladder}
\begin{tikzcd}[column sep={4em,between origins},row sep={4em,between origins}]
	&&&&&& {X_{k}} & {} \\
	&&&&& {X_{k-1}} && {Y_{k}} \\
	&&& {} & {X_{k-2}} && {Y_{k-1}} \\
	&&& \iddots && {Y_{k-2}} \\
	{} && {X_0} && \iddots \\
	& {Z} && {Y_0}
	\arrow["a", from=1-7, to=2-8]
	\arrow[from=2-6, to=3-7]
	\arrow[from=2-6, to=1-7]
	\arrow[from=3-7, to=2-8]
	\arrow[from=3-5, to=2-6]
	\arrow[from=3-5, to=4-6]
	\arrow[from=4-6, to=3-7]
	\arrow[from=6-2, to=5-3]
	\arrow[from=5-3, to=6-4]
	\arrow[from=5-3, to=4-4]
	\arrow[from=6-4, to=5-5]
	\arrow[from=4-4, to=3-5]
	\arrow[from=5-5, to=4-6]
\end{tikzcd}
\end{equation}
A region of $\AR(Q)$ which corresponds to a ladder in $\AR(Q^{\textrm{op}})$ will be called a \emph{dual ladder}.
The following lemma can be proven by a straightforward induction on $k$.

\begin{lemma}\label{lem:ladder}
For a ladder in $\text{AR}(Q)$ labeled as in \eqref{eq:ladder},
we have $Z\simeq \ker a$.
\end{lemma}

\subsection{Stability and Auslander-Reiten quivers}
To move our examination of stability from the representation category to the Auslander-Reiten quiver, we are led to introduce some additional terminology and reductions.

\begin{definition}\label{def:positive}
Let $\mu$ be a stability function on a quiver $Q$.  
We say that \emph{$\mu$ is positive} on a nonzero map $f\colon A \to B$ in $\rep(Q)$ if $\mu(A)<\mu(B)$.  We extend this terminology by saying $\mu$ is positive on a short exact sequence with three nonzero terms when $\mu$ is positive on one of its nonzero maps (and thus, on both by the see-saw property).

We further say that an arrow in AR($Q$) is \emph{positive} if the corresponding irreducible morphism is positive.
\end{definition}

When $\mu$ is fixed, as it typically will be throughout the paper, we may simply say that a map, arrow, or short exact sequence itself is positive to mean $\mu$ is positive on it.
With this in mind, we can reframe the notion of total stability in terms of the Auslander-Reiten quiver.  We keep the following lemma in mind for the remainder of the paper.

\begin{lemma}\label{lem:stabilityarrows}
Let $\mu$ be a stability function on a Dynkin quiver $Q$.  Then $\mu$ is totally stable if and only if $\mu$ is positive on each arrow of $\AR(Q)$.
\end{lemma}
\begin{proof}
Suppose that $\mu$ is totally stable and $f\colon A \to B$ is an arrow of $\AR(Q)$.  Then $f$ is either injective or surjective, thus positivity follows from $B$ being $\mu$-stable in the first case and from $A$ being $\mu$-stable in the second case (and the see-saw property).

Now suppose that $\mu$ is positive on each arrow of $\AR(Q)$ and let $V\in \rep(Q)$ be indecomposable.  By Lemma \ref{lem:indecomp}, it is enough to consider $0<W<V$ with $W$ also indecomposable.  Since there exists a nonzero morphism from $W$ to $V$, there exists a path
\begin{equation}
W=X_0 \to X_1 \to X_2 \to \cdots \to X_l=V
\end{equation}
in $\AR(Q)$ (although this path itself may not represent the inclusion $W< V$).
This implies that $\mu(X_i) < \mu(X_{i+1})$ for all $i$, so $\mu(W) < \mu (V)$, completing the proof.
\end{proof}

\begin{definition}\label{def:borderseq}
An almost split sequence $0 \to \tau V \to E \to V \to 0$ with $E$ indecomposable is called a \emph{border sequence}.
\end{definition}

This terminology is meant to be intuitive with respect to the standard way of drawing Auslander-Reiten quivers on the page in Dynkin type (see, for example, \cite[Ch.~3]{Schiffler:2014aa}).  With this new terminology and Lemma \ref{lem:stabilityarrows}, we can restate Theorem \ref{thm:main} as saying that a stability function on a Dynkin quiver is totally stable if and only if it is positive on all border sequences.

\section{Proof of main theorem in type \texorpdfstring{$\mathbb{D}$}{D}}\label{sec:typeD}

We begin in type $\mathbb{D}$ because it contains all of the essential techniques we need; then proofs in types $\mathbb{A}$ and $\mathbb{E}$  will follow easy but careful modifications.  
Let $Q$ be a type $\mathbb{D}$ quiver and $\mu\colon \rep^*(Q)\to \RR$ a stability function on $Q$ which is positive on all border sequences. 
We label the vertices of a type $\mathbb{D}$ quiver $1, 2, \dotsc, n$ so that 1 and 2 are at the ends of the short branches, 3 is the branch point, and $4, \dotsc, n$ count outward from the branch point (see the top of Figure \ref{fig:typeDAR} for an example).
Then an example of an Auslander-Reiten quiver of type $\mathbb{D}_8$ is shown at the bottom of Figure \ref{fig:typeDAR}.

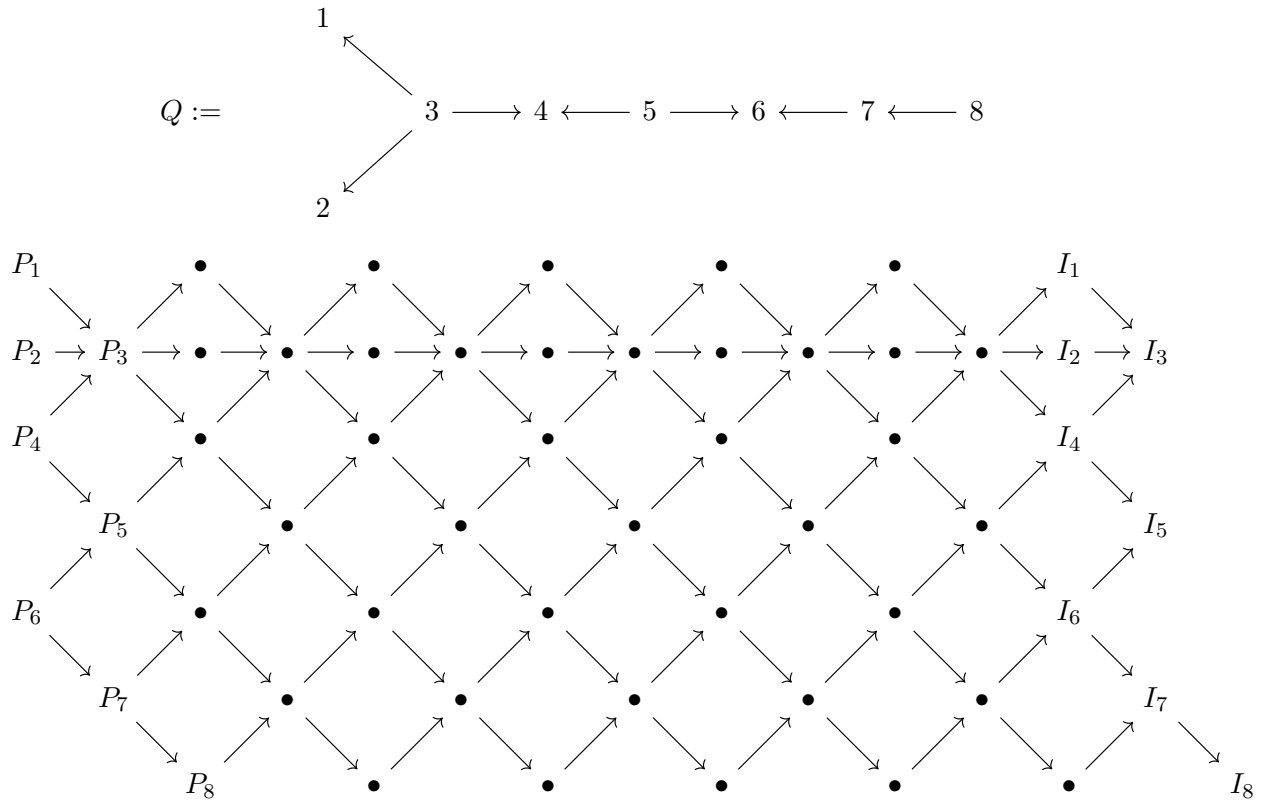
\begin{figure}[H]
\[\begin{tikzcd}
	& 1 \\
	{Q:=} & {} & 3 & 4 & 5 & 6 & 7 & 8 \\
	& 2
	\arrow[from=2-3, to=1-2]
	\arrow[from=2-3, to=3-2]
	\arrow[from=2-5, to=2-4]
	\arrow[from=2-5, to=2-6]
	\arrow[from=2-7, to=2-6]
	\arrow[from=2-8, to=2-7]
	\arrow[from=2-3, to=2-4]
\end{tikzcd}\]
\[\begin{tikzcd}[column sep={3em,between origins},row sep={3em,between origins}]
	{P_1} & \textcolor{white}{\bullet} & \bullet && \bullet && \bullet && \bullet && \bullet && {I_1} &&&& \textcolor{white}{\bullet} & \textcolor{white}{\bullet} \\
	{P_2} & {P_3} & \bullet & \bullet & \bullet & \bullet & \bullet & \bullet & \bullet & \bullet & \bullet & \bullet & {I_2} & {I_3} &&& \textcolor{white}{\bullet} & \textcolor{white}{\bullet} \\
	{P_4} & \textcolor{white}{\bullet} & \bullet && \bullet && \bullet && \bullet && \bullet && {I_4} &&&& \textcolor{white}{\bullet} & \textcolor{white}{\bullet} \\
	{} & {P_5} & \textcolor{white}{\bullet} & \bullet && \bullet && \bullet && \bullet && \bullet & \textcolor{white}{\bullet} & {I_5} &&& \textcolor{white}{\bullet} & \textcolor{white}{\bullet} \\
	{P_6} & \textcolor{white}{\bullet} & \bullet && \bullet && \bullet && \bullet && \bullet && {I_6} &&&& \textcolor{white}{\bullet} & \textcolor{white}{\bullet} \\
	& {P_7} & \textcolor{white}{\bullet} & \bullet && \bullet && \bullet && \bullet && \bullet & \textcolor{white}{\bullet} & {I_7} &&& \textcolor{white}{\bullet} & \textcolor{white}{\bullet} \\
	& \textcolor{white}{\bullet} & {P_8} && \bullet && \bullet && \bullet && \bullet && \bullet && {I_8} && \textcolor{white}{\bullet}
	\arrow[from=1-3, to=2-4]
	\arrow[from=2-3, to=2-4]
	\arrow[from=3-3, to=2-4]
	\arrow[from=3-3, to=4-4]
	\arrow[from=5-3, to=4-4]
	\arrow[from=5-3, to=6-4]
	\arrow[from=7-3, to=6-4]
	\arrow[from=6-4, to=5-5]
	\arrow[from=5-5, to=4-6]
	\arrow[from=4-6, to=3-7]
	\arrow[from=3-7, to=2-8]
	\arrow[from=4-4, to=3-5]
	\arrow[from=3-5, to=4-6]
	\arrow[from=4-4, to=5-5]
	\arrow[from=2-4, to=3-5]
	\arrow[from=2-4, to=2-5]
	\arrow[from=2-4, to=1-5]
	\arrow[from=1-5, to=2-6]
	\arrow[from=2-5, to=2-6]
	\arrow[from=3-5, to=2-6]
	\arrow[from=2-6, to=3-7]
	\arrow[from=2-6, to=2-7]
	\arrow[from=2-6, to=1-7]
	\arrow[from=1-7, to=2-8]
	\arrow[from=2-7, to=2-8]
	\arrow[from=2-8, to=2-9]
	\arrow[from=2-8, to=1-9]
	\arrow[from=1-11, to=2-12]
	\arrow[from=2-11, to=2-12]
	\arrow[from=6-4, to=7-5]
	\arrow[from=2-2, to=1-3]
	\arrow[from=2-2, to=2-3]
	\arrow[from=2-2, to=3-3]
	\arrow[from=4-2, to=3-3]
	\arrow[from=6-2, to=5-3]
	\arrow[from=4-2, to=5-3]
	\arrow[from=6-2, to=7-3]
	\arrow[from=1-1, to=2-2]
	\arrow[from=2-1, to=2-2]
	\arrow[from=3-1, to=2-2]
	\arrow[from=3-1, to=4-2]
	\arrow[from=5-1, to=4-2]
	\arrow[from=5-1, to=6-2]
	\arrow[from=2-12, to=1-13]
	\arrow[from=2-12, to=2-13]
	\arrow[from=2-12, to=3-13]
	\arrow[from=3-13, to=2-14]
	\arrow[from=2-13, to=2-14]
	\arrow[from=1-13, to=2-14]
	\arrow[from=3-13, to=4-14]
	\arrow[from=7-5, to=6-6]
	\arrow[from=6-6, to=5-7]
	\arrow[from=5-7, to=4-8]
	\arrow[from=4-8, to=3-9]
	\arrow[from=3-9, to=2-10]
	\arrow[from=2-8, to=3-9]
	\arrow[from=2-9, to=2-10]
	\arrow[from=1-9, to=2-10]
	\arrow[from=3-7, to=4-8]
	\arrow[from=4-6, to=5-7]
	\arrow[from=5-5, to=6-6]
	\arrow[from=2-10, to=2-11]
	\arrow[from=2-10, to=1-11]
	\arrow[from=6-6, to=7-7]
	\arrow[from=5-7, to=6-8]
	\arrow[from=4-8, to=5-9]
	\arrow[from=3-9, to=4-10]
	\arrow[from=2-10, to=3-11]
	\arrow[from=3-11, to=2-12]
	\arrow[from=6-8, to=7-9]
	\arrow[from=7-7, to=6-8]
	\arrow[from=6-8, to=5-9]
	\arrow[from=5-9, to=6-10]
	\arrow[from=7-9, to=6-10]
	\arrow[from=5-9, to=4-10]
	\arrow[from=4-10, to=3-11]
	\arrow[from=4-12, to=3-13]
	\arrow[from=5-11, to=4-12]
	\arrow[from=4-10, to=5-11]
	\arrow[from=6-10, to=5-11]
	\arrow[from=6-10, to=7-11]
	\arrow[from=5-11, to=6-12]
	\arrow[from=7-11, to=6-12]
	\arrow[from=6-12, to=5-13]
	\arrow[from=4-12, to=5-13]
	\arrow[from=3-11, to=4-12]
	\arrow[from=5-13, to=4-14]
	\arrow[from=6-12, to=7-13]
	\arrow[from=7-13, to=6-14]
	\arrow[from=5-13, to=6-14]
	\arrow[from=6-14, to=7-15]
\end{tikzcd}\]
\caption{A quiver of type $\mathbb{D}_8$ and its Auslander-Reiten quiver}
\label{fig:typeDAR}
\end{figure}

To prove that $\mu$ is totally stable, we go through various regions of $\AR(Q)$ in a carefully chosen order to show that our assumption implies each arrow is positive.  The key region is in the following definition.

\begin{definition}
A vertex $x$ of AR($Q$) is in the \emph{pyramid region} if there exist paths in AR($Q$) from $P_n$ to $x$ and from $x$ to $I_n$.
An arrow is in the pyramid region if both its source and target are in the pyramid.
A vertex $x$ is in the \emph{left wall of the pyramid} if there is a unique path from $P_n$ to $x$ in $\AR(Q)$.
An arrow $a$ is in the \emph{left wall of the pyramid} if its source and target are.
We define vertices and arrows in the \emph{right wall of the pyramid} similarly.
\end{definition}

The pyramid region for our running example is illustrated with blue, squiggly arrows and starred vertices in Figure \ref{fig:typeDpyramid2}. Vertices in the pyramid region (besides $P_8$ and $I_8$) are labeled with $\star$ here, and arrows in the pyramid region are blue and squiggly. 
Notice that it always contains exactly one triple mesh at the top, using that $\dim\Hom_Q(P_n, I_n)=1$ and standard computations in Auslander-Reiten quivers \cite[\S3.3.4.1]{Schiffler:2014aa}.
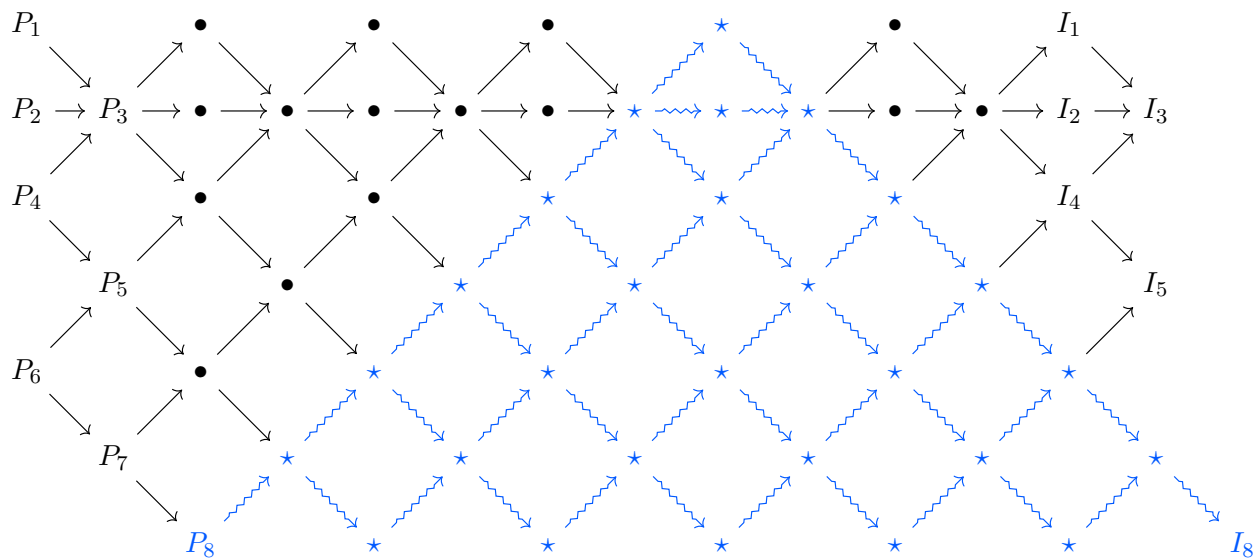
\begin{figure}[h]
\[\begin{tikzcd}[column sep={3em,between origins},row sep={3em,between origins}]
	{P_1} & \textcolor{white}{\bullet} & \bullet && \bullet && \bullet && \textcolor{rgb,255:red,0;green,89;blue,255}{\star} && \bullet && {I_1} &&&& \textcolor{white}{\bullet} & \textcolor{white}{\bullet} \\
	{P_2} & {P_3} & \bullet & \bullet & \bullet & \bullet & \bullet & \textcolor{rgb,255:red,0;green,89;blue,255}{\star} & \textcolor{rgb,255:red,0;green,89;blue,255}{\star} & \textcolor{rgb,255:red,0;green,89;blue,255}{\star} & \bullet & \bullet & {I_2} & {I_3} &&& \textcolor{white}{\bullet} & \textcolor{white}{\bullet} \\
	{P_4} & \textcolor{white}{\bullet} & \bullet && \bullet && \textcolor{rgb,255:red,0;green,89;blue,255}{\star} && \textcolor{rgb,255:red,0;green,89;blue,255}{\star} && \textcolor{rgb,255:red,0;green,89;blue,255}{\star} && {I_4} &&&& \textcolor{white}{\bullet} & \textcolor{white}{\bullet} \\
	{} & {P_5} & \textcolor{white}{\bullet} & \bullet && \textcolor{rgb,255:red,0;green,89;blue,255}{\star} && \textcolor{rgb,255:red,0;green,89;blue,255}{\star} && \textcolor{rgb,255:red,0;green,89;blue,255}{\star} && \textcolor{rgb,255:red,0;green,89;blue,255}{\star} & \textcolor{white}{\bullet} & {I_5} &&& \textcolor{white}{\bullet} & \textcolor{white}{\bullet} \\
	{P_6} & \textcolor{white}{\bullet} & \bullet && \textcolor{rgb,255:red,0;green,89;blue,255}{\star} && \textcolor{rgb,255:red,0;green,89;blue,255}{\star} && \textcolor{rgb,255:red,0;green,89;blue,255}{\star} && \textcolor{rgb,255:red,0;green,89;blue,255}{\star} && \textcolor{rgb,255:red,0;green,89;blue,255}{\star} &&&& \textcolor{white}{\bullet} & \textcolor{white}{\bullet} \\
	& {P_7} & \textcolor{white}{\bullet} & \textcolor{rgb,255:red,0;green,89;blue,255}{\star} && \textcolor{rgb,255:red,0;green,89;blue,255}{\star} && \textcolor{rgb,255:red,0;green,89;blue,255}{\star} && \textcolor{rgb,255:red,0;green,89;blue,255}{\star} && \textcolor{rgb,255:red,0;green,89;blue,255}{\star} & \textcolor{white}{\bullet} & \textcolor{rgb,255:red,0;green,89;blue,255}{\star} &&& \textcolor{white}{\bullet} & \textcolor{white}{\bullet} \\
	& \textcolor{white}{\bullet} & \textcolor{rgb,255:red,0;green,89;blue,255}{P_8} && \textcolor{rgb,255:red,0;green,89;blue,255}{\star} && \textcolor{rgb,255:red,0;green,89;blue,255}{\star} && \textcolor{rgb,255:red,0;green,89;blue,255}{\star} && \textcolor{rgb,255:red,0;green,89;blue,255}{\star} && \textcolor{rgb,255:red,0;green,89;blue,255}{\star} && \textcolor{rgb,255:red,0;green,89;blue,255}{I_8} && \textcolor{white}{\bullet}
	\arrow[from=1-3, to=2-4]
	\arrow[from=2-3, to=2-4]
	\arrow[from=3-3, to=2-4]
	\arrow[from=3-3, to=4-4]
	\arrow[from=5-3, to=4-4]
	\arrow[from=5-3, to=6-4]
	\arrow[color={rgb,255:red,0;green,89;blue,255}, squiggly, from=7-3, to=6-4]
	\arrow[color={rgb,255:red,0;green,89;blue,255}, squiggly, from=6-4, to=5-5]
	\arrow[color={rgb,255:red,0;green,89;blue,255}, squiggly, from=5-5, to=4-6]
	\arrow[color={rgb,255:red,0;green,89;blue,255}, squiggly, from=4-6, to=3-7]
	\arrow[color={rgb,255:red,0;green,89;blue,255}, squiggly, from=3-7, to=2-8]
	\arrow[from=4-4, to=3-5]
	\arrow[from=3-5, to=4-6]
	\arrow[from=4-4, to=5-5]
	\arrow[from=2-4, to=3-5]
	\arrow[from=2-4, to=2-5]
	\arrow[from=2-4, to=1-5]
	\arrow[from=1-5, to=2-6]
	\arrow[from=2-5, to=2-6]
	\arrow[from=3-5, to=2-6]
	\arrow[from=2-6, to=3-7]
	\arrow[from=2-6, to=2-7]
	\arrow[from=2-6, to=1-7]
	\arrow[from=1-7, to=2-8]
	\arrow[from=2-7, to=2-8]
	\arrow[color={rgb,255:red,0;green,89;blue,255}, squiggly, from=2-8, to=2-9]
	\arrow[color={rgb,255:red,0;green,89;blue,255}, squiggly, from=2-8, to=1-9]
	\arrow[from=1-11, to=2-12]
	\arrow[from=2-11, to=2-12]
	\arrow[color={rgb,255:red,0;green,89;blue,255}, squiggly, from=6-4, to=7-5]
	\arrow[from=2-2, to=1-3]
	\arrow[from=2-2, to=2-3]
	\arrow[from=2-2, to=3-3]
	\arrow[from=4-2, to=3-3]
	\arrow[from=6-2, to=5-3]
	\arrow[from=4-2, to=5-3]
	\arrow[from=6-2, to=7-3]
	\arrow[from=1-1, to=2-2]
	\arrow[from=2-1, to=2-2]
	\arrow[from=3-1, to=2-2]
	\arrow[from=3-1, to=4-2]
	\arrow[from=5-1, to=4-2]
	\arrow[from=5-1, to=6-2]
	\arrow[from=2-12, to=1-13]
	\arrow[from=2-12, to=2-13]
	\arrow[from=2-12, to=3-13]
	\arrow[from=3-13, to=2-14]
	\arrow[from=2-13, to=2-14]
	\arrow[from=1-13, to=2-14]
	\arrow[from=3-13, to=4-14]
	\arrow[color={rgb,255:red,0;green,89;blue,255}, squiggly, from=7-5, to=6-6]
	\arrow[color={rgb,255:red,0;green,89;blue,255}, squiggly, from=6-6, to=5-7]
	\arrow[color={rgb,255:red,0;green,89;blue,255}, squiggly, from=5-7, to=4-8]
	\arrow[color={rgb,255:red,0;green,89;blue,255}, squiggly, from=4-8, to=3-9]
	\arrow[color={rgb,255:red,0;green,89;blue,255}, squiggly, from=3-9, to=2-10]
	\arrow[color={rgb,255:red,0;green,89;blue,255}, squiggly, from=2-8, to=3-9]
	\arrow[color={rgb,255:red,0;green,89;blue,255}, squiggly, from=2-9, to=2-10]
	\arrow[color={rgb,255:red,0;green,89;blue,255}, squiggly, from=1-9, to=2-10]
	\arrow[color={rgb,255:red,0;green,89;blue,255}, squiggly, from=3-7, to=4-8]
	\arrow[color={rgb,255:red,0;green,89;blue,255}, squiggly, from=4-6, to=5-7]
	\arrow[color={rgb,255:red,0;green,89;blue,255}, squiggly, from=5-5, to=6-6]
	\arrow[from=2-10, to=2-11]
	\arrow[from=2-10, to=1-11]
	\arrow[color={rgb,255:red,0;green,89;blue,255}, squiggly, from=6-6, to=7-7]
	\arrow[color={rgb,255:red,0;green,89;blue,255}, squiggly, from=5-7, to=6-8]
	\arrow[color={rgb,255:red,0;green,89;blue,255}, squiggly, from=4-8, to=5-9]
	\arrow[color={rgb,255:red,0;green,89;blue,255}, squiggly, from=3-9, to=4-10]
	\arrow[color={rgb,255:red,0;green,89;blue,255}, squiggly, from=2-10, to=3-11]
	\arrow[from=3-11, to=2-12]
	\arrow[color={rgb,255:red,0;green,89;blue,255}, squiggly, from=6-8, to=7-9]
	\arrow[color={rgb,255:red,0;green,89;blue,255}, squiggly, from=7-7, to=6-8]
	\arrow[color={rgb,255:red,0;green,89;blue,255}, squiggly, from=6-8, to=5-9]
	\arrow[color={rgb,255:red,0;green,89;blue,255}, squiggly, from=5-9, to=6-10]
	\arrow[color={rgb,255:red,0;green,89;blue,255}, squiggly, from=7-9, to=6-10]
	\arrow[color={rgb,255:red,0;green,89;blue,255}, squiggly, from=5-9, to=4-10]
	\arrow[color={rgb,255:red,0;green,89;blue,255}, squiggly, from=4-10, to=3-11]
	\arrow[from=4-12, to=3-13]
	\arrow[color={rgb,255:red,0;green,89;blue,255}, squiggly, from=5-11, to=4-12]
	\arrow[color={rgb,255:red,0;green,89;blue,255}, squiggly, from=4-10, to=5-11]
	\arrow[color={rgb,255:red,0;green,89;blue,255}, squiggly, from=6-10, to=5-11]
	\arrow[color={rgb,255:red,0;green,89;blue,255}, squiggly, from=6-10, to=7-11]
	\arrow[color={rgb,255:red,0;green,89;blue,255}, squiggly, from=5-11, to=6-12]
	\arrow[color={rgb,255:red,0;green,89;blue,255}, squiggly, from=7-11, to=6-12]
	\arrow[color={rgb,255:red,0;green,89;blue,255}, squiggly, from=6-12, to=5-13]
	\arrow[color={rgb,255:red,0;green,89;blue,255}, squiggly, from=4-12, to=5-13]
	\arrow[color={rgb,255:red,0;green,89;blue,255}, squiggly, from=3-11, to=4-12]
	\arrow[from=5-13, to=4-14]
	\arrow[color={rgb,255:red,0;green,89;blue,255}, squiggly, from=6-12, to=7-13]
	\arrow[color={rgb,255:red,0;green,89;blue,255}, squiggly, from=7-13, to=6-14]
	\arrow[color={rgb,255:red,0;green,89;blue,255}, squiggly, from=5-13, to=6-14]
	\arrow[color={rgb,255:red,0;green,89;blue,255}, squiggly, from=6-14, to=7-15]
\end{tikzcd}\]
\caption{The pyramid region for the running example}
\label{fig:typeDpyramid2}
\end{figure}
We begin with arrows which have source or target in the bottom row of the pyramid region.  Since all of these arrows are part of a border sequence, the following lemma holds by assumption.

\begin{lemma}\label{lem:bottomrow}
All arrows in the pyramid region of AR($Q$) with either source or target in the $\tau$-orbit of $P_n$ are positive. 
\end{lemma}

We next consider arrows with source and target both in the top two rows of vertices of AR($Q$), as presented in Figure \ref{fig:typeDAR}.

\begin{lemma}\label{lem:toprows}
All arrows with both source and target in the $\tau$-orbits of $P_1$, $P_2$, $P_3$ are positive.
\end{lemma}
\begin{proof}
All but two such arrows are part of a border sequence and thus positive by assumption.  The locations of the remaining two arrows in AR($Q$) depend on the orientation of $Q$.  There are two cases up to duality (replacing $Q$ with $Q^{\op}$), which we consider separately.

\smallskip

\noindent \emph{Case 1.} Assume the arrows incident to vertices 1 and 2 are both oriented towards the branch point, as shown below (unoriented edges can point in either direction).
\begin{equation}
\begin{tikzcd}
	1 \\
	& 3 & 4 & {n-1} & n \\
	2
	\arrow[from=1-1, to=2-2]
	\arrow[from=3-1, to=2-2]
	\arrow[no head, from=2-2, to=2-3]
	\arrow[no head, from=2-4, to=2-5]
	\arrow["\cdots"{marking}, draw=none, from=2-3, to=2-4]
\end{tikzcd}
\end{equation}
The $\tau$-orbits of $P_1$, $P_2$, and $P_3$ in AR($Q$) are arranged as follows.
\begin{equation}\label{eq:Dcase1}
\begin{tikzcd}
	& {P_1} && \bullet & \bullet && {I_1 \text{ or } I_2} \\
	{P_3} & {P_2} & \bullet & \bullet & \bullet & \bullet & {I_1 \text{ or } I_2}
	\arrow[from=2-6, to=1-7]
	\arrow[from=2-6, to=2-7]
	\arrow["\beta"', dashed, from=2-1, to=2-2]
	\arrow["\alpha", dashed, from=2-1, to=1-2]
	\arrow[from=1-2, to=2-3]
	\arrow[from=2-2, to=2-3]
	\arrow[from=2-3, to=2-4]
	\arrow["\cdots"{marking, allow upside down}, draw=none, from=2-4, to=2-5]
	\arrow[from=2-5, to=2-6]
	\arrow[from=2-3, to=1-4]
	\arrow[from=1-5, to=2-6]
	\arrow["\cdots"{marking, allow upside down}, draw=none, from=1-4, to=1-5]
\end{tikzcd}
\end{equation}
Here, the arrows which are positive by assumption are solid, and the two remaining arrows which we need to consider are dashed and labeled $\alpha, \beta$.

Now from standard descriptions of projective and injective representations in terms of paths (see for example \cite[\S 2.1]{Schiffler:2014aa}), we obtain short exact sequences:
\begin{equation}\label{eq:Dcase1seq}
\begin{split}
    &0 \to P_3 \xto{\alpha} P_1 \to I_1 \to 0,\\
    &0 \to P_3 \xto{\beta} P_2 \to I_2 \to 0.
\end{split}
\end{equation}
To put it another way, in this orientation we have $P_3\simeq \rad P_1 \simeq \rad P_2$, and both $I_1$ and $I_2$ are simple.
Then we see in \eqref{eq:Dcase1} that there is a path of positive arrows from $P_1$ to $I_1$, so $\mu(P_1) < \mu(I_1)$, and thus from \eqref{eq:Dcase1seq} we get $\alpha$ is positive as well.  Then $\beta$ is also positive by replacing vertex $1$ with $2$ in this proof.

\smallskip

\noindent \emph{Case 2.} Assume the arrow incident to vertex 1 points towards the branch point, and the arrow incident to vertex 2 points towards that vertex, as shown in (\ref{eq:typeDexample}) (unoriented edges can point in either direction).

\begin{equation}\label{eq:typeDexample}
\begin{tikzcd}
	1 \\
	& 3 & 4 & {n-1} & n \\
	2
	\arrow[from=1-1, to=2-2]
	\arrow[no head, from=2-2, to=2-3]
	\arrow[no head, from=2-4, to=2-5]
	\arrow["\cdots"{marking}, draw=none, from=2-3, to=2-4]
	\arrow[from=2-2, to=3-1]
\end{tikzcd}
\end{equation}

The $\tau$-orbits of $P_1$, $P_2$, and $P_3$ in AR($Q$) are arranged as follows:
\begin{equation}\label{eq:Dcase2}
\begin{tikzcd}
	{P_2} && \bullet && \bullet & \bullet && {I_2} \\
	& {P_3} & {P_1} & \bullet & \bullet & \bullet & \bullet & \bullet & {I_3} & {I_1}
	\arrow[from=1-1, to=2-2]
	\arrow[from=2-2, to=1-3]
	\arrow["\alpha"', dashed, from=2-2, to=2-3]
	\arrow[from=2-3, to=2-4]
	\arrow[from=1-3, to=2-4]
	\arrow[from=2-4, to=2-5]
	\arrow[from=2-6, to=2-7]
	\arrow["\cdots"{description}, draw=none, from=2-5, to=2-6]
	\arrow[from=2-7, to=1-8]
	\arrow["\beta", dashed, from=1-8, to=2-9]
	\arrow[from=2-7, to=2-8]
	\arrow[from=2-8, to=2-9]
	\arrow[from=2-9, to=2-10]
	\arrow[from=2-4, to=1-5]
	\arrow[from=1-6, to=2-7]
	\arrow["\cdots"{marking, allow upside down}, draw=none, from=1-5, to=1-6]
\end{tikzcd}
\end{equation}
Here, the arrows which are positive by assumption are solid, and the two remaining arrows which we need to consider, are dashed and labeled $\alpha, \beta$.

Similar to the last case, we have short exact sequences:
\begin{equation}\label{eq:Dcase2seq}
\begin{split}
        &0 \to P_3 \xto{\alpha} P_1 \to I_1 \to 0,\\
    &0 \to P_2 \to I_2 \xto{\beta} I_3 \to 0.
\end{split}
\end{equation}
Following solid paths within AR($Q$) and using these short exact sequences then again shows that both $\alpha$ and $\beta$ are positive.
\end{proof}

\begin{figure}\label{fig:typeDAR2}
\[\begin{tikzcd}[column sep={3em,between origins},row sep={3em,between origins}]
	{P_1} & \textcolor{white}{\bullet} & \bullet && \bullet && \bullet && \bullet && \bullet && {I_1} &&&& \textcolor{white}{\bullet} & \textcolor{white}{\bullet} \\
	{P_2} & {P_3} & \bullet & \bullet & \bullet & \bullet & \bullet & \bullet & \bullet & \bullet & \bullet & \bullet & {I_2} & {I_3} &&& \textcolor{white}{\bullet} & \textcolor{white}{\bullet} \\
	{P_4} & \textcolor{white}{\bullet} & \bullet && \bullet && \bullet && \bullet && \bullet && {I_4} &&&& \textcolor{white}{\bullet} & \textcolor{white}{\bullet} \\
	{} & {P_5} & \textcolor{white}{\bullet} & \bullet && \bullet && \bullet && \bullet && \bullet & \textcolor{white}{\bullet} & {I_5} &&& \textcolor{white}{\bullet} & \textcolor{white}{\bullet} \\
	{P_6} & \textcolor{white}{\bullet} & \bullet && \bullet && \bullet && \bullet && \bullet && {I_6} &&&& \textcolor{white}{\bullet} & \textcolor{white}{\bullet} \\
	& {P_7} & \textcolor{white}{\bullet} & \bullet && \bullet && \bullet && \bullet && \bullet & \textcolor{white}{\bullet} & {I_7} &&& \textcolor{white}{\bullet} & \textcolor{white}{\bullet} \\
	& \textcolor{white}{\bullet} & {P_8} && \bullet && \bullet && \bullet && \bullet && \bullet && {I_8} && \textcolor{white}{\bullet}
	\arrow[color={rgb,255:red,255;green,51;blue,78}, from=1-3, to=2-4]
	\arrow[color={rgb,255:red,255;green,51;blue,78}, from=2-3, to=2-4]
	\arrow[from=3-3, to=2-4]
	\arrow[from=3-3, to=4-4]
	\arrow[from=5-3, to=4-4]
	\arrow[from=5-3, to=6-4]
	\arrow[color={rgb,255:red,255;green,51;blue,78}, from=7-3, to=6-4]
	\arrow[from=6-4, to=5-5]
	\arrow[from=5-5, to=4-6]
	\arrow[from=4-6, to=3-7]
	\arrow[from=3-7, to=2-8]
	\arrow[from=4-4, to=3-5]
	\arrow[from=3-5, to=4-6]
	\arrow[from=4-4, to=5-5]
	\arrow[from=2-4, to=3-5]
	\arrow[color={rgb,255:red,255;green,51;blue,78}, from=2-4, to=2-5]
	\arrow[color={rgb,255:red,255;green,51;blue,78}, from=2-4, to=1-5]
	\arrow[color={rgb,255:red,255;green,51;blue,78}, from=1-5, to=2-6]
	\arrow[color={rgb,255:red,255;green,51;blue,78}, from=2-5, to=2-6]
	\arrow[from=3-5, to=2-6]
	\arrow[from=2-6, to=3-7]
	\arrow[color={rgb,255:red,255;green,51;blue,78}, from=2-6, to=2-7]
	\arrow[color={rgb,255:red,255;green,51;blue,78}, from=2-6, to=1-7]
	\arrow[color={rgb,255:red,255;green,51;blue,78}, from=1-7, to=2-8]
	\arrow[color={rgb,255:red,255;green,51;blue,78}, from=2-7, to=2-8]
	\arrow[color={rgb,255:red,255;green,51;blue,78}, from=2-8, to=2-9]
	\arrow[color={rgb,255:red,255;green,51;blue,78}, from=2-8, to=1-9]
	\arrow[color={rgb,255:red,255;green,51;blue,78}, from=1-11, to=2-12]
	\arrow[color={rgb,255:red,255;green,51;blue,78}, from=2-11, to=2-12]
	\arrow[color={rgb,255:red,255;green,51;blue,78}, from=6-4, to=7-5]
	\arrow[color={rgb,255:red,255;green,51;blue,78}, from=2-2, to=1-3]
	\arrow[color={rgb,255:red,255;green,51;blue,78}, from=2-2, to=2-3]
	\arrow[from=2-2, to=3-3]
	\arrow[from=4-2, to=3-3]
	\arrow[from=6-2, to=5-3]
	\arrow[from=4-2, to=5-3]
	\arrow[from=6-2, to=7-3]
	\arrow[color={rgb,255:red,255;green,51;blue,78}, from=1-1, to=2-2]
	\arrow[color={rgb,255:red,255;green,51;blue,78}, from=2-1, to=2-2]
	\arrow[from=3-1, to=2-2]
	\arrow[from=3-1, to=4-2]
	\arrow[from=5-1, to=4-2]
	\arrow[from=5-1, to=6-2]
	\arrow[color={rgb,255:red,255;green,51;blue,78}, from=2-12, to=1-13]
	\arrow[color={rgb,255:red,255;green,51;blue,78}, from=2-12, to=2-13]
	\arrow[from=2-12, to=3-13]
	\arrow[from=3-13, to=2-14]
	\arrow[color={rgb,255:red,255;green,51;blue,78}, from=2-13, to=2-14]
	\arrow[color={rgb,255:red,255;green,51;blue,78}, from=1-13, to=2-14]
	\arrow[from=3-13, to=4-14]
	\arrow[color={rgb,255:red,255;green,51;blue,78}, from=7-5, to=6-6]
	\arrow[from=6-6, to=5-7]
	\arrow[from=5-7, to=4-8]
	\arrow[from=4-8, to=3-9]
	\arrow[from=3-9, to=2-10]
	\arrow[from=2-8, to=3-9]
	\arrow[color={rgb,255:red,255;green,51;blue,78}, from=2-9, to=2-10]
	\arrow[color={rgb,255:red,255;green,51;blue,78}, from=1-9, to=2-10]
	\arrow[from=3-7, to=4-8]
	\arrow[from=4-6, to=5-7]
	\arrow[from=5-5, to=6-6]
	\arrow[color={rgb,255:red,255;green,51;blue,78}, from=2-10, to=2-11]
	\arrow[color={rgb,255:red,255;green,51;blue,78}, from=2-10, to=1-11]
	\arrow[color={rgb,255:red,255;green,51;blue,78}, from=6-6, to=7-7]
	\arrow[from=5-7, to=6-8]
	\arrow[from=4-8, to=5-9]
	\arrow[from=3-9, to=4-10]
	\arrow[from=2-10, to=3-11]
	\arrow[from=3-11, to=2-12]
	\arrow[color={rgb,255:red,255;green,51;blue,78}, from=6-8, to=7-9]
	\arrow[color={rgb,255:red,255;green,51;blue,78}, from=7-7, to=6-8]
	\arrow[from=6-8, to=5-9]
	\arrow[from=5-9, to=6-10]
	\arrow[color={rgb,255:red,255;green,51;blue,78}, from=7-9, to=6-10]
	\arrow[from=5-9, to=4-10]
	\arrow[from=4-10, to=3-11]
	\arrow[from=4-12, to=3-13]
	\arrow[from=5-11, to=4-12]
	\arrow[from=4-10, to=5-11]
	\arrow[from=6-10, to=5-11]
	\arrow[color={rgb,255:red,255;green,51;blue,78}, from=6-10, to=7-11]
	\arrow[from=5-11, to=6-12]
	\arrow[color={rgb,255:red,255;green,51;blue,78}, from=7-11, to=6-12]
	\arrow[from=6-12, to=5-13]
	\arrow[from=4-12, to=5-13]
	\arrow[from=3-11, to=4-12]
	\arrow[from=5-13, to=4-14]
	\arrow[color={rgb,255:red,255;green,51;blue,78}, from=6-12, to=7-13]
	\arrow[color={rgb,255:red,255;green,51;blue,78}, from=7-13, to=6-14]
	\arrow[from=5-13, to=6-14]
	\arrow[color={rgb,255:red,255;green,51;blue,78}, from=6-14, to=7-15]
\end{tikzcd}\]
\caption{The solid red arrows are positive due to Lemma \ref{lem:bottomrow} and Lemma \ref{lem:toprows}.}
\end{figure}
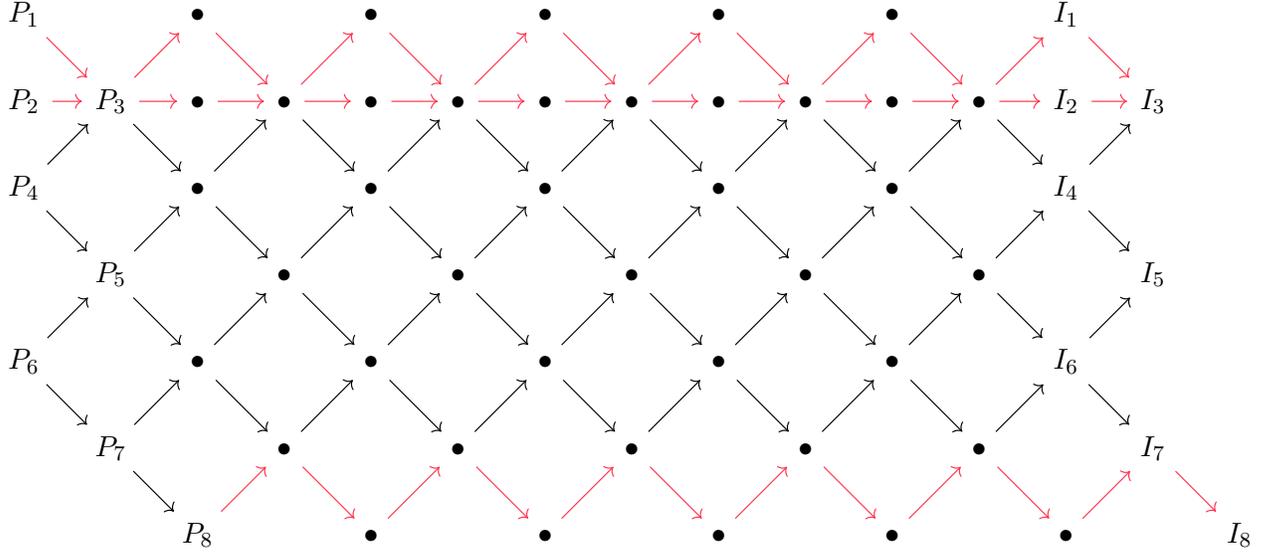

We can now deal with the whole pyramid region.

\begin{lemma}\label{lem:typeDpyramid}
With the setup above, every arrow in the pyramid region is positive.
\end{lemma}

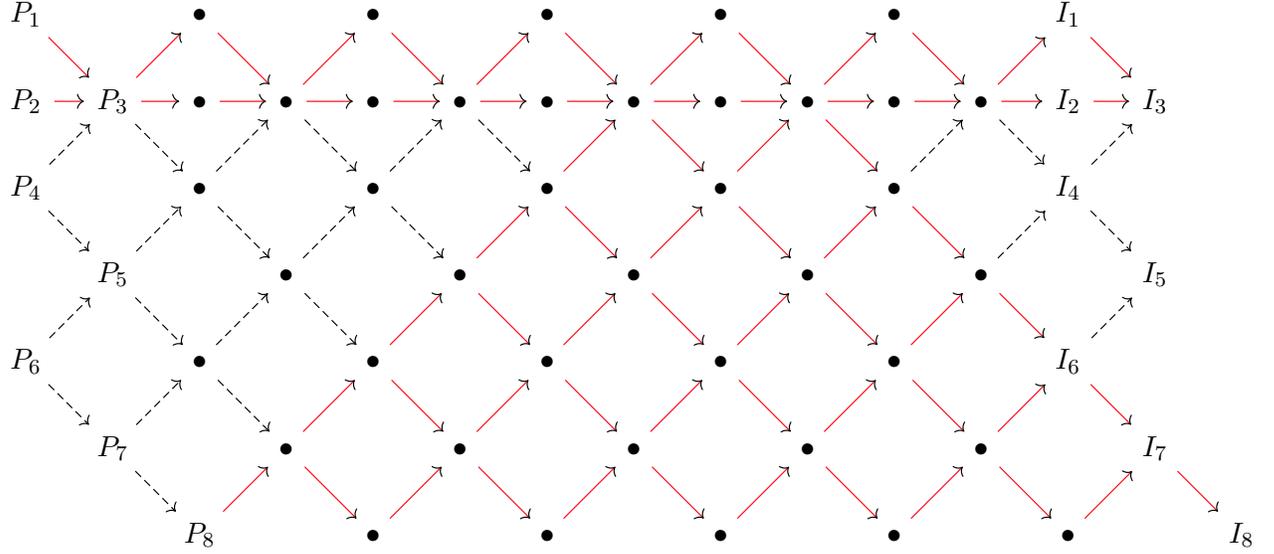
\begin{figure}\label{fig:typeDAR3}
\[\begin{tikzcd}[column sep={3em,between origins},row sep={3em,between origins}]
	{P_1} & \textcolor{white}{\bullet} & \bullet && \bullet && \bullet && \bullet && \bullet && {I_1} &&&& \textcolor{white}{\bullet} & \textcolor{white}{\bullet} \\
	{P_2} & {P_3} & \bullet & \bullet & \bullet & \bullet & \bullet & \bullet & \bullet & \bullet & \bullet & \bullet & {I_2} & {I_3} &&& \textcolor{white}{\bullet} & \textcolor{white}{\bullet} \\
	{P_4} & \textcolor{white}{\bullet} & \bullet && \bullet & {} & \bullet && \bullet && \bullet && {I_4} &&&& \textcolor{white}{\bullet} & \textcolor{white}{\bullet} \\
	{} & {P_5} & \textcolor{white}{\bullet} & \bullet && \bullet && \bullet && \bullet && \bullet & \textcolor{white}{\bullet} & {I_5} &&& \textcolor{white}{\bullet} & \textcolor{white}{\bullet} \\
	{P_6} & \textcolor{white}{\bullet} & \bullet && \bullet && \bullet && \bullet && \bullet && {I_6} &&&& \textcolor{white}{\bullet} & \textcolor{white}{\bullet} \\
	& {P_7} & \textcolor{white}{\bullet} & \bullet && \bullet && \bullet && \bullet && \bullet & \textcolor{white}{\bullet} & {I_7} &&& \textcolor{white}{\bullet} & \textcolor{white}{\bullet} \\
	& \textcolor{white}{\bullet} & {P_8} && \bullet && \bullet && \bullet && \bullet && \bullet && {I_8} && \textcolor{white}{\bullet}
	\arrow[draw={rgb,255:red,255;green,5;blue,18}, from=1-3, to=2-4]
	\arrow[draw={rgb,255:red,255;green,5;blue,18}, from=2-3, to=2-4]
	\arrow[dashed, from=3-3, to=2-4]
	\arrow[dashed, from=3-3, to=4-4]
	\arrow[dashed, from=5-3, to=4-4]
	\arrow[dashed, from=5-3, to=6-4]
	\arrow[draw={rgb,255:red,255;green,5;blue,18}, from=7-3, to=6-4]
	\arrow[draw={rgb,255:red,255;green,5;blue,18}, from=6-4, to=5-5]
	\arrow[draw={rgb,255:red,255;green,5;blue,18}, from=5-5, to=4-6]
	\arrow[draw={rgb,255:red,255;green,5;blue,18}, from=4-6, to=3-7]
	\arrow[draw={rgb,255:red,255;green,5;blue,18}, from=3-7, to=2-8]
	\arrow[dashed, from=4-4, to=3-5]
	\arrow[dashed, from=3-5, to=4-6]
	\arrow[dashed, from=4-4, to=5-5]
	\arrow[dashed, from=2-4, to=3-5]
	\arrow[draw={rgb,255:red,255;green,5;blue,18}, from=2-4, to=2-5]
	\arrow[draw={rgb,255:red,255;green,5;blue,18}, from=2-4, to=1-5]
	\arrow[draw={rgb,255:red,255;green,5;blue,18}, from=1-5, to=2-6]
	\arrow[draw={rgb,255:red,255;green,5;blue,18}, from=2-5, to=2-6]
	\arrow[dashed, from=3-5, to=2-6]
	\arrow[dashed, from=2-6, to=3-7]
	\arrow[draw={rgb,255:red,255;green,5;blue,18}, from=2-6, to=2-7]
	\arrow[draw={rgb,255:red,255;green,5;blue,18}, from=2-6, to=1-7]
	\arrow[draw={rgb,255:red,255;green,5;blue,18}, from=1-7, to=2-8]
	\arrow[draw={rgb,255:red,255;green,5;blue,18}, from=2-7, to=2-8]
	\arrow[draw={rgb,255:red,255;green,5;blue,18}, from=2-8, to=2-9]
	\arrow[draw={rgb,255:red,255;green,5;blue,18}, from=2-8, to=1-9]
	\arrow[draw={rgb,255:red,255;green,5;blue,18}, from=1-11, to=2-12]
	\arrow[draw={rgb,255:red,255;green,5;blue,18}, from=2-11, to=2-12]
	\arrow[draw={rgb,255:red,255;green,5;blue,18}, from=6-4, to=7-5]
	\arrow[draw={rgb,255:red,255;green,5;blue,18}, from=2-2, to=1-3]
	\arrow[draw={rgb,255:red,255;green,5;blue,18}, from=2-2, to=2-3]
	\arrow[dashed, from=2-2, to=3-3]
	\arrow[dashed, from=4-2, to=3-3]
	\arrow[dashed, from=6-2, to=5-3]
	\arrow[dashed, from=4-2, to=5-3]
	\arrow[dashed, from=6-2, to=7-3]
	\arrow[draw={rgb,255:red,255;green,5;blue,18}, from=1-1, to=2-2]
	\arrow[draw={rgb,255:red,255;green,5;blue,18}, from=2-1, to=2-2]
	\arrow[dashed, from=3-1, to=2-2]
	\arrow[dashed, from=3-1, to=4-2]
	\arrow[dashed, from=5-1, to=4-2]
	\arrow[dashed, from=5-1, to=6-2]
	\arrow[draw={rgb,255:red,255;green,5;blue,18}, from=2-12, to=1-13]
	\arrow[draw={rgb,255:red,255;green,5;blue,18}, from=2-12, to=2-13]
	\arrow[dashed, from=2-12, to=3-13]
	\arrow[dashed, from=3-13, to=2-14]
	\arrow[draw={rgb,255:red,255;green,5;blue,18}, from=2-13, to=2-14]
	\arrow[draw={rgb,255:red,255;green,5;blue,18}, from=1-13, to=2-14]
	\arrow[dashed, from=3-13, to=4-14]
	\arrow[draw={rgb,255:red,255;green,5;blue,18}, from=7-5, to=6-6]
	\arrow[draw={rgb,255:red,255;green,5;blue,18}, from=6-6, to=5-7]
	\arrow[draw={rgb,255:red,255;green,5;blue,18}, from=5-7, to=4-8]
	\arrow[draw={rgb,255:red,255;green,5;blue,18}, from=4-8, to=3-9]
	\arrow[draw={rgb,255:red,255;green,5;blue,18}, from=3-9, to=2-10]
	\arrow[draw={rgb,255:red,255;green,5;blue,18}, from=2-8, to=3-9]
	\arrow[draw={rgb,255:red,255;green,5;blue,18}, from=2-9, to=2-10]
	\arrow[draw={rgb,255:red,255;green,5;blue,18}, from=1-9, to=2-10]
	\arrow[draw={rgb,255:red,255;green,5;blue,18}, from=3-7, to=4-8]
	\arrow[draw={rgb,255:red,255;green,5;blue,18}, from=4-6, to=5-7]
	\arrow[draw={rgb,255:red,255;green,5;blue,18}, from=5-5, to=6-6]
	\arrow[draw={rgb,255:red,255;green,5;blue,18}, from=2-10, to=2-11]
	\arrow[draw={rgb,255:red,255;green,5;blue,18}, from=2-10, to=1-11]
	\arrow[draw={rgb,255:red,255;green,5;blue,18}, from=6-6, to=7-7]
	\arrow[draw={rgb,255:red,255;green,5;blue,18}, from=5-7, to=6-8]
	\arrow[draw={rgb,255:red,255;green,5;blue,18}, from=4-8, to=5-9]
	\arrow[draw={rgb,255:red,255;green,5;blue,18}, from=3-9, to=4-10]
	\arrow[draw={rgb,255:red,255;green,5;blue,18}, from=2-10, to=3-11]
	\arrow[dashed, from=3-11, to=2-12]
	\arrow[draw={rgb,255:red,255;green,5;blue,18}, from=6-8, to=7-9]
	\arrow[draw={rgb,255:red,255;green,5;blue,18}, from=7-7, to=6-8]
	\arrow[draw={rgb,255:red,255;green,5;blue,18}, from=6-8, to=5-9]
	\arrow[draw={rgb,255:red,255;green,5;blue,18}, from=5-9, to=6-10]
	\arrow[draw={rgb,255:red,255;green,5;blue,18}, from=7-9, to=6-10]
	\arrow[draw={rgb,255:red,255;green,5;blue,18}, from=5-9, to=4-10]
	\arrow[draw={rgb,255:red,255;green,5;blue,18}, from=4-10, to=3-11]
	\arrow[dashed, from=4-12, to=3-13]
	\arrow[draw={rgb,255:red,255;green,5;blue,18}, from=5-11, to=4-12]
	\arrow[draw={rgb,255:red,255;green,5;blue,18}, from=4-10, to=5-11]
	\arrow[draw={rgb,255:red,255;green,5;blue,18}, from=6-10, to=5-11]
	\arrow[draw={rgb,255:red,255;green,5;blue,18}, from=6-10, to=7-11]
	\arrow[draw={rgb,255:red,255;green,5;blue,18}, from=5-11, to=6-12]
	\arrow[draw={rgb,255:red,255;green,5;blue,18}, from=7-11, to=6-12]
	\arrow[draw={rgb,255:red,255;green,5;blue,18}, from=6-12, to=5-13]
	\arrow[draw={rgb,255:red,255;green,5;blue,18}, from=4-12, to=5-13]
	\arrow[draw={rgb,255:red,255;green,5;blue,18}, from=3-11, to=4-12]
	\arrow[dashed, from=5-13, to=4-14]
	\arrow[draw={rgb,255:red,255;green,5;blue,18}, from=6-12, to=7-13]
	\arrow[draw={rgb,255:red,255;green,5;blue,18}, from=7-13, to=6-14]
	\arrow[draw={rgb,255:red,255;green,5;blue,18}, from=5-13, to=6-14]
	\arrow[draw={rgb,255:red,255;green,5;blue,18}, from=6-14, to=7-15]
\end{tikzcd}\]
\caption{The solid red arrows are positive by results through Lemma \ref{lem:typeDpyramid}.}
\end{figure}

\begin{proof}
We perform induction upwards on the rows of arrows in the pyramid region. Let $X \xto{a} Y$ be an arrow in the pyramid region.  By Lemma \ref{lem:toprows}, we may assume $X$ and $Y$ are not both in the $\tau$-orbits of $P_1, P_2, P_3$.
Replacing $Q$ with $Q^{\mathrm{op}}$ if necessary, without loss of generality we may assume $X$ is in the $\tau$-orbit of $P_i$ and $Y$ is in the $\tau$-orbit of $P_{i+1}$ for some $i\geq 3$, meaning that $a$ points downwards (and rightwards) in Figure \ref{fig:typeDAR}.
For the base case, if $i+1=n$ then $a$ is positive by Lemma \ref{lem:bottomrow}. 

If $i+1<n$, then $a$ is the top rung of a ladder as in \eqref{eq:ladder}, with $X=X_k$ and $Y=Y_k$ in that diagram, and we obtain a short exact sequence $0 \to Z \to X \xto{a} Y \to 0$ by Lemma \ref{lem:ladder}.  Since this ladder is contained entirely in the pyramid, we have by induction that the arrows in lower rows are positive, so 
\begin{equation}
    \mu(Z) < \mu(X_0) < \mu (Y_0) < \cdots < \mu(Y_k)=\mu(Y).
\end{equation}
By the see-saw property, $a$ is positive.
\end{proof}

We next will show that arrows ``perpendicular'' to the pyramid walls are positive.  As before, we just deal with the region to the left of the pyramid, and the region to the right is obtained by duality.
We formalize this in the following notation:

\begin{notation}\label{not:bi}
For $4 \leq i \leq n$, let $b_i$ be the unique arrow  in $AR(Q)$ (if such an arrow exists) which is not in the pyramid, but has target on the pyramid wall in the $\tau$-orbit of $P_i$.
We say that an arrow of the form $\tau^k b_i$, for some $k \geq 0$, is \emph{perpendicular to the left wall of the pyramid}. 

For $4 \leq i \leq n$, let $a_i$ be the unique arrow in $AR(Q)$ which is in the left wall of the pyramid with source in the $\tau$-orbit of $P_i$.
An arrow of the form $\tau^k a_i$, for some $k \geq 0$, is \emph{parallel to the left wall of the pyramid}.
\end{notation}

To avoid introducing more technical notation for ranges of the indices, when using the notations $\tau^k a_i, \tau^k b_i$ we always assume that $i$ and $k$ range over those values for which the corresponding arrow exists, depending on the orientation of $Q$.

A region of arrows perpendicular and parallel to the pyramid wall is illustrated in Figure \ref{fig:perp&par}. The squiggly blue arrows belong to the pyramid, with the arrows $a_i$ belonging to the left pyramid wall, and the solid black arrows are either parallel or perpendicular to the pyramid.

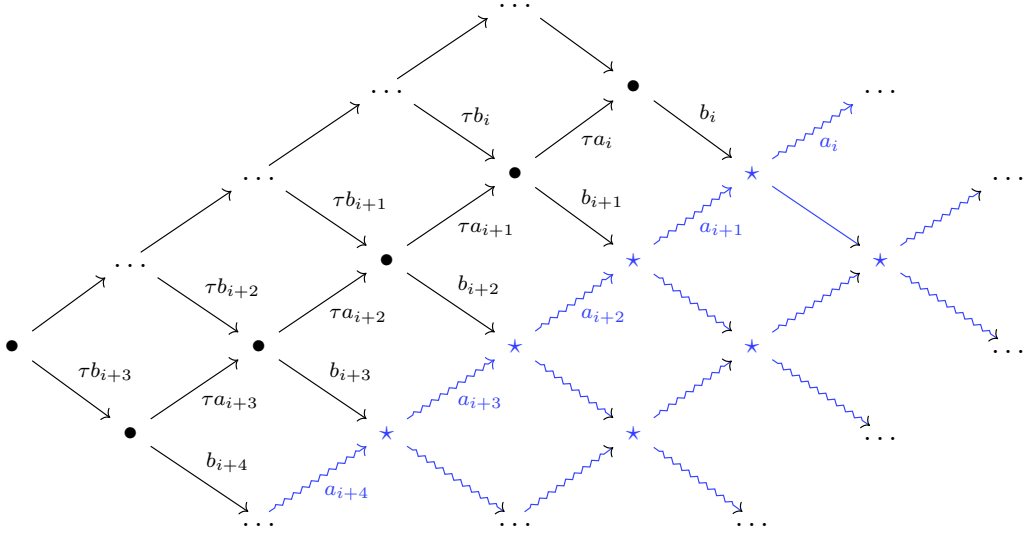
\begin{figure}[h]
\[\begin{tikzcd}
	\textcolor{white}{\bullet} & \textcolor{white}{\bullet} & \textcolor{white}{\bullet} & \textcolor{white}{\bullet} & \dots & \textcolor{white}{\bullet} & \textcolor{white}{\bullet} & \textcolor{white}{\bullet} & \textcolor{white}{\bullet} \\
	\textcolor{white}{\bullet} & \textcolor{white}{\bullet} & \textcolor{white}{\bullet} & \dots & \textcolor{white}{\bullet} & \bullet & \textcolor{white}{\bullet} & \dots & \textcolor{white}{\bullet} \\
	\textcolor{white}{\bullet} & \textcolor{white}{\bullet} & \dots & \textcolor{white}{\bullet} & \bullet & \textcolor{white}{\bullet} & \textcolor{rgb,255:red,51;green,65;blue,255}{\star} & \textcolor{white}{\bullet} & \dots \\
	\textcolor{white}{\bullet} & \dots & \textcolor{white}{\bullet} & \bullet & \textcolor{white}{\bullet} & \textcolor{rgb,255:red,51;green,65;blue,255}{\star} & \textcolor{white}{\bullet} & \textcolor{rgb,255:red,51;green,65;blue,255}{\star} & \textcolor{white}{\bullet} \\
	\bullet & \textcolor{white}{\bullet} & \bullet & \textcolor{white}{\bullet} & \textcolor{rgb,255:red,51;green,65;blue,255}{\star} & \textcolor{white}{\bullet} & \textcolor{rgb,255:red,51;green,65;blue,255}{\star} & \textcolor{white}{\bullet} & \dots \\
	\textcolor{white}{\bullet} & \bullet & \textcolor{white}{\bullet} & \textcolor{rgb,255:red,51;green,65;blue,255}{\star} & \textcolor{white}{\bullet} & \textcolor{rgb,255:red,51;green,65;blue,255}{\star} & \textcolor{white}{\bullet} & \dots & \textcolor{white}{\bullet} \\
	\textcolor{white}{\bullet} & \textcolor{white}{\bullet} & \dots & \textcolor{white}{\bullet} & \dots & \textcolor{white}{\bullet} & \dots & \textcolor{white}{\bullet} & \textcolor{white}{\bullet}
	\arrow["{a_{i+4}}"', color={rgb,255:red,51;green,65;blue,255}, squiggly, from=7-3, to=6-4]
	\arrow["{a_{i+3}}"', color={rgb,255:red,51;green,65;blue,255}, squiggly, from=6-4, to=5-5]
	\arrow["{a_{i+2}}"', color={rgb,255:red,51;green,65;blue,255}, squiggly, from=5-5, to=4-6]
	\arrow["{a_{i+1}}"', color={rgb,255:red,51;green,65;blue,255}, squiggly, from=4-6, to=3-7]
	\arrow[draw={rgb,255:red,51;green,65;blue,255}, squiggly, from=5-5, to=6-6]
	\arrow[draw={rgb,255:red,51;green,65;blue,255}, squiggly, from=4-6, to=5-7]
	\arrow[draw={rgb,255:red,51;green,65;blue,255}, from=3-7, to=4-8]
	\arrow["{b_{i+3}}", from=5-3, to=6-4]
	\arrow["{b_{i+2}}", from=4-4, to=5-5]
	\arrow["{b_{i+1}}", from=3-5, to=4-6]
	\arrow["{\tau a_{i+2}}"', from=5-3, to=4-4]
	\arrow["{\tau a_{i+1}}"', from=4-4, to=3-5]
	\arrow["{\tau b_{i+1}}", from=3-3, to=4-4]
	\arrow["{\tau b_{i}}", from=2-4, to=3-5]
	\arrow["{a_i}"', color={rgb,255:red,51;green,65;blue,255}, squiggly, from=3-7, to=2-8]
	\arrow[draw={rgb,255:red,51;green,65;blue,255}, squiggly, from=4-8, to=3-9]
	\arrow[draw={rgb,255:red,51;green,65;blue,255}, squiggly, from=4-8, to=5-9]
	\arrow[draw={rgb,255:red,51;green,65;blue,255}, squiggly, from=5-7, to=6-8]
	\arrow[draw={rgb,255:red,51;green,65;blue,255}, squiggly, from=6-6, to=7-7]
	\arrow[draw={rgb,255:red,51;green,65;blue,255}, squiggly, from=6-4, to=7-5]
	\arrow[draw={rgb,255:red,51;green,65;blue,255}, squiggly, from=7-5, to=6-6]
	\arrow[draw={rgb,255:red,51;green,65;blue,255}, squiggly, from=6-6, to=5-7]
	\arrow[draw={rgb,255:red,51;green,65;blue,255}, squiggly, from=5-7, to=4-8]
	\arrow["{b_{i}}", from=2-6, to=3-7]
	\arrow["{\tau a_{i}}"', from=3-5, to=2-6]
	\arrow["{\tau b_{i+2}}", from=4-2, to=5-3]
	\arrow[from=4-2, to=3-3]
	\arrow[from=3-3, to=2-4]
	\arrow[from=2-4, to=1-5]
	\arrow[from=1-5, to=2-6]
	\arrow["{b_{i+4}}", from=6-2, to=7-3]
	\arrow["{\tau b_{i+3}}", from=5-1, to=6-2]
	\arrow[from=5-1, to=4-2]
	\arrow["{\tau a_{i+3}}"', from=6-2, to=5-3]
\end{tikzcd}\]
\caption{A region of arrows perpendicular and parallel to the pyramid wall}
\label{fig:perp&par}
\end{figure}

The following lemma is key to identifying the cokernel of an arrow perpendicular to the pyramid wall.
We remark that for equioriented type $\mathbb{D}$ quivers (i.e. all arrows oriented towards vertex $n$), there is no arrow $b_4$, so Lemma \ref{lem:triplemeshcoker} is vacuously true due to our notation convention, though its dual would be applied to deal with the region to the right of the pyramid.

\begin{lemma}\label{lem:triplemeshcoker}
Consider the arrow $b_4$ defined in Notation \ref{not:bi} and the neighboring region of $AR(Q)$ shown below.
\begin{equation}\label{eq:triplemeshes}
\begin{tikzcd}
	& {W_1} && {Y_1} \\
	V & {W_2} & X & {Y_2} & Z \\
	& {W_3} && {Y_3}
	\arrow["{b_4}"', from=2-1, to=3-2]
	\arrow[from=3-2, to=2-3]
	\arrow[from=2-1, to=2-2]
	\arrow[from=2-2, to=2-3]
	\arrow[from=2-1, to=1-2]
	\arrow[from=1-2, to=2-3]
	\arrow[from=2-3, to=1-4]
	\arrow[from=2-3, to=2-4]
	\arrow[from=2-3, to=3-4]
	\arrow["{b_4'}"', from=3-4, to=2-5]
	\arrow[from=2-4, to=2-5]
	\arrow[from=1-4, to=2-5]
\end{tikzcd}
\end{equation}
We have that $\coker b_4 \simeq \coker b_4' \simeq I_n$.
\end{lemma}
\begin{proof}
Since $b_4$ and $b_4'$ are irreducible morphisms, each is either injective with indecomposable cokernel, or surjective with indecomposable kernel (see \S\ref{sec:AR}).  It follows from the computation below that they are both injective. For a Dynkin quiver, indecomposables are determined by their dimension vectors, so it is enough to show that
$\coker b_4$ and $\coker b_4'$ have the same dimension vector.  This is illustrated by the following computation which uses basic properties of exact sequences in Auslander-Reiten quivers.  At each step, we simply replace terms more towards the left in the Auslander-Reiten quiver with terms appearing more towards the right.
\begin{equation}\label{eq:Dcokercalc}
    \begin{split}
    \ddim \coker b_4 &= \ddim W_3 - \ddim V = \ddim X - \ddim W_1 - \ddim W_2\\
    &= \ddim Y_1 + \ddim Y_2 - \ddim X = \ddim Z - \ddim Y_3 =\ddim \coker b_4'.
    \end{split}
\end{equation}
Now translating down the right wall of the pyramid using the dual version of Lemma \ref{lem:ladder}, we find that $\coker b'_4 \simeq I_n$.
\end{proof}

\begin{lemma}\label{lem:pyramidperp}
The arrow $b_i$ is positive for $4 \leq i \leq n$. \end{lemma}
\begin{proof}
By Lemma \ref{lem:triplemeshcoker} and the fact that 
$\coker b_i \simeq \coker b_4$ for all $i$ by Lemma \ref{lem:ladder},
we get a short exact sequence
\begin{equation}\label{eq:pyramidperp}
    0 \to A_i \xto{b_i} B_i \to I_n \to 0.
\end{equation}
Since $B_i$ is on the left wall of the pyramid and $I_n$ is on the right wall of the pyramid, there is a path inside the pyramid from $B_i$ to $I_n$.
This path is positive by Lemma \ref{lem:typeDpyramid}.  Thus $\mu(B_i) < \mu(I_n)$, implying $\mu(A_i) < \mu(B_i)$ by the see-saw property, so $b_i$ is positive.
\end{proof}

\begin{figure}
\[\begin{tikzcd}
	X & \textcolor{white}{\bullet} & W & \textcolor{white}{\bullet} & \textcolor{white}{\bullet} & \textcolor{white}{\bullet} & \bullet & \textcolor{white}{\bullet} & \dots \\
	\textcolor{white}{\bullet} & Y & \textcolor{white}{\bullet} & \textcolor{white}{\bullet} & \textcolor{white}{\bullet} & \bullet & \textcolor{white}{\bullet} & \textcolor{rgb,255:red,51;green,65;blue,255}{\star} & \textcolor{white}{\bullet} \\
	\textcolor{white}{\bullet} & \textcolor{white}{\bullet} & \bullet & \textcolor{white}{\bullet} & \textcolor{white}{\bullet} & \textcolor{white}{\bullet} & \textcolor{rgb,255:red,51;green,65;blue,255}{\star} & \textcolor{white}{\bullet} & \textcolor{white}{\bullet} \\
	\textcolor{white}{\bullet} & \textcolor{white}{\bullet} && \ddots & \textcolor{white}{\bullet} & \dots & \textcolor{white}{\bullet} & \textcolor{white}{\bullet} & \textcolor{white}{\bullet} \\
	\textcolor{white}{\bullet} & \textcolor{white}{\bullet} & \textcolor{white}{\bullet} & \textcolor{white}{\bullet} & \textcolor{rgb,255:red,51;green,65;blue,255}{\star} & \textcolor{white}{\bullet} & \textcolor{white}{\bullet} & \textcolor{white}{\bullet} & \textcolor{white}{\bullet}
	\arrow["{a_{i+1}}"', color={rgb,255:red,51;green,65;blue,255}, squiggly, from=3-7, to=2-8]
	\arrow["{b_{i+1}}", from=2-6, to=3-7]
	\arrow["{b_i}", from=1-7, to=2-8]
	\arrow["{\tau^k b_i}"', from=1-1, to=2-2]
	\arrow["{\tau^{k-1}b_{i+1}}"', from=2-2, to=3-3]
	\arrow["{\tau^{k-2}b_{i+2}}"', from=3-3, to=4-4]
	\arrow["{b_{i+k}}"', from=4-4, to=5-5]
	\arrow["{a_{i+k}}"', color={rgb,255:red,51;green,65;blue,255}, squiggly, from=5-5, to=4-6]
	\arrow["{a_{i+2}}"', color={rgb,255:red,51;green,65;blue,255}, squiggly, from=4-6, to=3-7]
	\arrow["{\tau(a_i)}", from=2-6, to=1-7]
	\arrow["{a_i}"', color={rgb,255:red,51;green,65;blue,255}, squiggly, from=2-8, to=1-9]
	\arrow["{\tau^ka_i}", from=2-2, to=1-3]
\end{tikzcd}\]
\caption{Visual aid for proofs of Lemmas \ref{lem:typeDperp} and \ref{lem:typeDparallel}}\label{fig:parallelproof}
\end{figure}
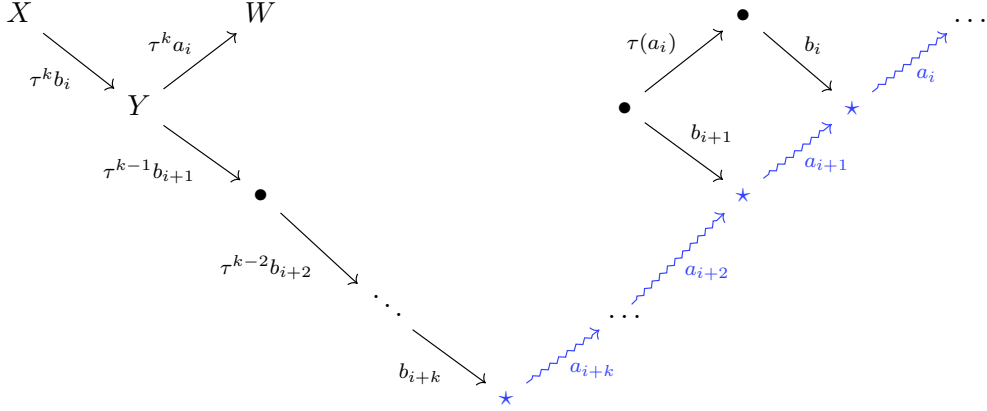

\begin{lemma}\label{lem:typeDperp}
All arrows $\tau^k b_i$ (i.e. those perpendicular to the left wall) are positive. 
\end{lemma}

\begin{proof}
We proceed by induction on the power of $\tau$. We have already proven the statement for the arrows $b_i$ in Lemma \ref{lem:pyramidperp}, so the base case holds.
Suppose we know that arrows of the form $\tau^{j}b_i$ are positive for all $0\leq j\leq k-1$. We show $\tau^{k}b_i$ is positive.

Since $\coker b_i \simeq I_n$ from \eqref{eq:pyramidperp}, we apply $\tau^k$ to get a short exact sequence
\begin{equation}
    0 \to X \xto{\tau^k b_i} Y \to \tau^k I_n \to 0.
\end{equation}
By the induction hypothesis, we have a positive path of arrows (see Figure \ref{fig:parallelproof}) from $Y$ to the wall of the pyramid given by  
\begin{equation}\label{eq:perppath}
(\tau^{k-1}b_{i+1})\ (\tau^{k-2}b_{i+2}) \cdots b_{i+k}.
\end{equation}
Letting $Z$ denote the target of $b_{i+k}$, this gives $\mu(Y) < \mu(Z)$.

By unraveling the notation and definitions of the pyramid, continuing the path from the target of $b_{i+k}$ directly down and the right we eventually end up at $\tau^{i+k-4} I_n$.  Since $i \geq 4$ by definition of the arrows $b_i$, there is furthermore a path from $\tau^{i+k-4} I_n$ to $\tau^k I_n$.  Now we have a path from $Z$ to $\tau^k I_n$ contained entirely within the pyramid, so by 
Lemma \ref{lem:typeDpyramid} we get $\mu(Z)<\mu(\tau^k I_n)$ and thus $\mu(Y)<\mu(\tau^k I_n)$ and eventually $\mu(X)<\mu(Y)$ by the see-saw property, meaning $\tau^k b_i$ is positive.
\end{proof}

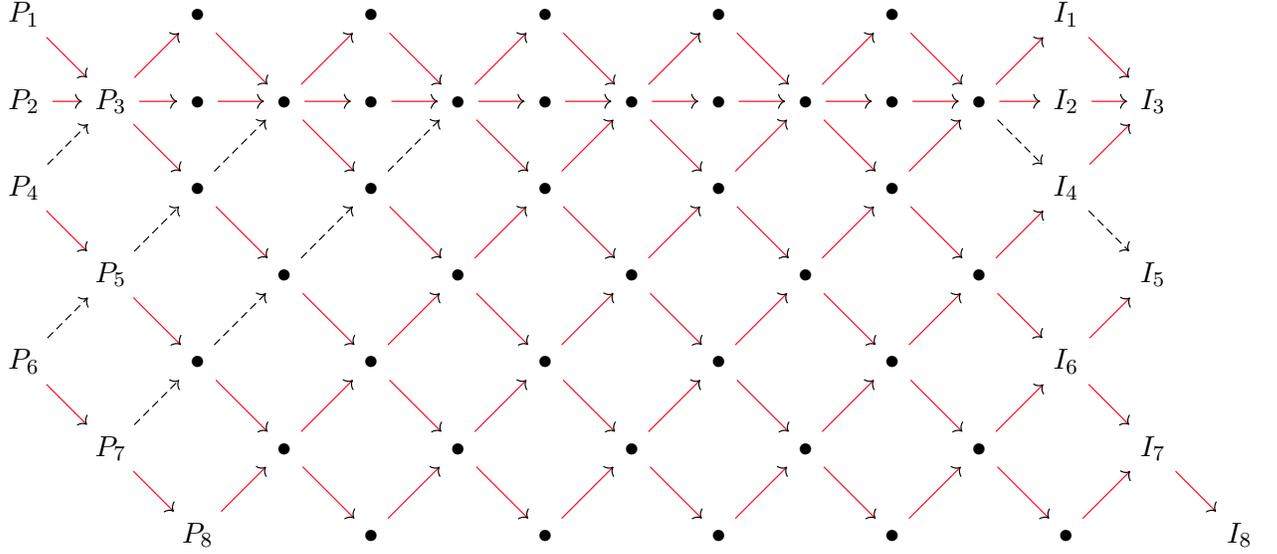
\begin{figure}[h]\label{fig:typeDAR4}
\[\begin{tikzcd}[column sep={3em,between origins},row sep={3em,between origins}]
	{P_1} & \textcolor{white}{\bullet} & \bullet && \bullet && \bullet && \bullet && \bullet && {I_1} &&&& \textcolor{white}{\bullet} & \textcolor{white}{\bullet} \\
	{P_2} & {P_3} & \bullet & \bullet & \bullet & \bullet & \bullet & \bullet & \bullet & \bullet & \bullet & \bullet & {I_2} & {I_3} &&& \textcolor{white}{\bullet} & \textcolor{white}{\bullet} \\
	{P_4} & \textcolor{white}{\bullet} & \bullet && \bullet & {} & \bullet && \bullet && \bullet && {I_4} &&&& \textcolor{white}{\bullet} & \textcolor{white}{\bullet} \\
	{} & {P_5} & \textcolor{white}{\bullet} & \bullet && \bullet && \bullet && \bullet && \bullet & \textcolor{white}{\bullet} & {I_5} &&& \textcolor{white}{\bullet} & \textcolor{white}{\bullet} \\
	{P_6} & \textcolor{white}{\bullet} & \bullet && \bullet && \bullet && \bullet && \bullet && {I_6} &&&& \textcolor{white}{\bullet} & \textcolor{white}{\bullet} \\
	& {P_7} & \textcolor{white}{\bullet} & \bullet && \bullet && \bullet && \bullet && \bullet & \textcolor{white}{\bullet} & {I_7} &&& \textcolor{white}{\bullet} & \textcolor{white}{\bullet} \\
	& \textcolor{white}{\bullet} & {P_8} && \bullet && \bullet && \bullet && \bullet && \bullet && {I_8} && \textcolor{white}{\bullet}
	\arrow[draw={rgb,255:red,255;green,5;blue,26}, from=1-3, to=2-4]
	\arrow[draw={rgb,255:red,255;green,5;blue,26}, from=2-3, to=2-4]
	\arrow[dashed, from=3-3, to=2-4]
	\arrow[draw={rgb,255:red,255;green,5;blue,26}, from=3-3, to=4-4]
	\arrow[dashed, from=5-3, to=4-4]
	\arrow[draw={rgb,255:red,255;green,5;blue,26}, from=5-3, to=6-4]
	\arrow[draw={rgb,255:red,255;green,5;blue,26}, from=7-3, to=6-4]
	\arrow[draw={rgb,255:red,255;green,5;blue,26}, from=6-4, to=5-5]
	\arrow[draw={rgb,255:red,255;green,5;blue,26}, from=5-5, to=4-6]
	\arrow[draw={rgb,255:red,255;green,5;blue,26}, from=4-6, to=3-7]
	\arrow[draw={rgb,255:red,255;green,5;blue,26}, from=3-7, to=2-8]
	\arrow[dashed, from=4-4, to=3-5]
	\arrow[draw={rgb,255:red,255;green,5;blue,26}, from=3-5, to=4-6]
	\arrow[draw={rgb,255:red,255;green,5;blue,26}, from=4-4, to=5-5]
	\arrow[draw={rgb,255:red,255;green,5;blue,26}, from=2-4, to=3-5]
	\arrow[draw={rgb,255:red,255;green,5;blue,26}, from=2-4, to=2-5]
	\arrow[draw={rgb,255:red,255;green,5;blue,26}, from=2-4, to=1-5]
	\arrow[draw={rgb,255:red,255;green,5;blue,26}, from=1-5, to=2-6]
	\arrow[draw={rgb,255:red,255;green,5;blue,26}, from=2-5, to=2-6]
	\arrow[dashed, from=3-5, to=2-6]
	\arrow[draw={rgb,255:red,255;green,5;blue,26}, from=2-6, to=3-7]
	\arrow[draw={rgb,255:red,255;green,5;blue,26}, from=2-6, to=2-7]
	\arrow[draw={rgb,255:red,255;green,5;blue,26}, from=2-6, to=1-7]
	\arrow[draw={rgb,255:red,255;green,5;blue,26}, from=1-7, to=2-8]
	\arrow[draw={rgb,255:red,255;green,5;blue,26}, from=2-7, to=2-8]
	\arrow[draw={rgb,255:red,255;green,5;blue,26}, from=2-8, to=2-9]
	\arrow[draw={rgb,255:red,255;green,5;blue,26}, from=2-8, to=1-9]
	\arrow[draw={rgb,255:red,255;green,5;blue,26}, from=1-11, to=2-12]
	\arrow[draw={rgb,255:red,255;green,5;blue,26}, from=2-11, to=2-12]
	\arrow[draw={rgb,255:red,255;green,5;blue,26}, from=6-4, to=7-5]
	\arrow[draw={rgb,255:red,255;green,5;blue,26}, from=2-2, to=1-3]
	\arrow[draw={rgb,255:red,255;green,5;blue,26}, from=2-2, to=2-3]
	\arrow[draw={rgb,255:red,255;green,5;blue,26}, from=2-2, to=3-3]
	\arrow[dashed, from=4-2, to=3-3]
	\arrow[dashed, from=6-2, to=5-3]
	\arrow[draw={rgb,255:red,255;green,5;blue,26}, from=4-2, to=5-3]
	\arrow[draw={rgb,255:red,255;green,5;blue,26}, from=6-2, to=7-3]
	\arrow[draw={rgb,255:red,255;green,5;blue,26}, from=1-1, to=2-2]
	\arrow[draw={rgb,255:red,255;green,5;blue,26}, from=2-1, to=2-2]
	\arrow[dashed, from=3-1, to=2-2]
	\arrow[draw={rgb,255:red,255;green,5;blue,26}, from=3-1, to=4-2]
	\arrow[dashed, from=5-1, to=4-2]
	\arrow[draw={rgb,255:red,255;green,5;blue,26}, from=5-1, to=6-2]
	\arrow[draw={rgb,255:red,255;green,5;blue,26}, from=2-12, to=1-13]
	\arrow[draw={rgb,255:red,255;green,5;blue,26}, from=2-12, to=2-13]
	\arrow[dashed, from=2-12, to=3-13]
	\arrow[draw={rgb,255:red,255;green,5;blue,26}, from=3-13, to=2-14]
	\arrow[draw={rgb,255:red,255;green,5;blue,26}, from=2-13, to=2-14]
	\arrow[draw={rgb,255:red,255;green,5;blue,26}, from=1-13, to=2-14]
	\arrow[dashed, from=3-13, to=4-14]
	\arrow[draw={rgb,255:red,255;green,5;blue,26}, from=7-5, to=6-6]
	\arrow[draw={rgb,255:red,255;green,5;blue,26}, from=6-6, to=5-7]
	\arrow[draw={rgb,255:red,255;green,5;blue,26}, from=5-7, to=4-8]
	\arrow[draw={rgb,255:red,255;green,5;blue,26}, from=4-8, to=3-9]
	\arrow[draw={rgb,255:red,255;green,5;blue,26}, from=3-9, to=2-10]
	\arrow[draw={rgb,255:red,255;green,5;blue,26}, from=2-8, to=3-9]
	\arrow[draw={rgb,255:red,255;green,5;blue,26}, from=2-9, to=2-10]
	\arrow[draw={rgb,255:red,255;green,5;blue,26}, from=1-9, to=2-10]
	\arrow[draw={rgb,255:red,255;green,5;blue,26}, from=3-7, to=4-8]
	\arrow[draw={rgb,255:red,255;green,5;blue,26}, from=4-6, to=5-7]
	\arrow[draw={rgb,255:red,255;green,5;blue,26}, from=5-5, to=6-6]
	\arrow[draw={rgb,255:red,255;green,5;blue,26}, from=2-10, to=2-11]
	\arrow[draw={rgb,255:red,255;green,5;blue,26}, from=2-10, to=1-11]
	\arrow[draw={rgb,255:red,255;green,5;blue,26}, from=6-6, to=7-7]
	\arrow[draw={rgb,255:red,255;green,5;blue,26}, from=5-7, to=6-8]
	\arrow[draw={rgb,255:red,255;green,5;blue,26}, from=4-8, to=5-9]
	\arrow[draw={rgb,255:red,255;green,5;blue,26}, from=3-9, to=4-10]
	\arrow[draw={rgb,255:red,255;green,5;blue,26}, from=2-10, to=3-11]
	\arrow[draw={rgb,255:red,255;green,5;blue,26}, from=3-11, to=2-12]
	\arrow[draw={rgb,255:red,255;green,5;blue,26}, from=6-8, to=7-9]
	\arrow[draw={rgb,255:red,255;green,5;blue,26}, from=7-7, to=6-8]
	\arrow[draw={rgb,255:red,255;green,5;blue,26}, from=6-8, to=5-9]
	\arrow[draw={rgb,255:red,255;green,5;blue,26}, from=5-9, to=6-10]
	\arrow[draw={rgb,255:red,255;green,5;blue,26}, from=7-9, to=6-10]
	\arrow[draw={rgb,255:red,255;green,5;blue,26}, from=5-9, to=4-10]
	\arrow[draw={rgb,255:red,255;green,5;blue,26}, from=4-10, to=3-11]
	\arrow[draw={rgb,255:red,255;green,5;blue,26}, from=4-12, to=3-13]
	\arrow[draw={rgb,255:red,255;green,5;blue,26}, from=5-11, to=4-12]
	\arrow[draw={rgb,255:red,255;green,5;blue,26}, from=4-10, to=5-11]
	\arrow[draw={rgb,255:red,255;green,5;blue,26}, from=6-10, to=5-11]
	\arrow[draw={rgb,255:red,255;green,5;blue,26}, from=6-10, to=7-11]
	\arrow[draw={rgb,255:red,255;green,5;blue,26}, from=5-11, to=6-12]
	\arrow[draw={rgb,255:red,255;green,5;blue,26}, from=7-11, to=6-12]
	\arrow[draw={rgb,255:red,255;green,5;blue,26}, from=6-12, to=5-13]
	\arrow[draw={rgb,255:red,255;green,5;blue,26}, from=4-12, to=5-13]
	\arrow[draw={rgb,255:red,255;green,5;blue,26}, from=3-11, to=4-12]
	\arrow[draw={rgb,255:red,255;green,5;blue,26}, from=5-13, to=4-14]
	\arrow[draw={rgb,255:red,255;green,5;blue,26}, from=6-12, to=7-13]
	\arrow[draw={rgb,255:red,255;green,5;blue,26}, from=7-13, to=6-14]
	\arrow[draw={rgb,255:red,255;green,5;blue,26}, from=5-13, to=6-14]
	\arrow[draw={rgb,255:red,255;green,5;blue,26}, from=6-14, to=7-15]
\end{tikzcd}\]
\caption{The solid red arrows are positive by results through Lemma \ref{lem:typeDperp}.}
\end{figure}

\begin{lemma}\label{lem:typeDparallel}
All arrows $\tau^k a_i$ (i.e. those parallel to the left wall) are positive.
\end{lemma}
\begin{proof}
Denoting the corresponding short exact sequence by
\begin{equation}
    0 \to Y \xto{\tau^k a_i} W \to V \to 0,
\end{equation}
we will show $\mu(Y) < \mu(V)$ by constructing paths from $Y$ to $s(a_{i+k})$ and from $s(a_{i+k})$ to $V$ which are each positive; again the reader may see Figure \ref{fig:parallelproof}.

Applying Lemma \ref{lem:typeDperp}, the path in \eqref{eq:perppath} from $Y$ to $s(a_{i+k})$ is positive.

Then the dual of Lemma \ref{lem:ladder} and associated dual ladder give a path from $s(a_{i+k})$ to $V\simeq \coker a_{i+k}$ which is entirely contained in the pyramid, thus positive by Lemma \ref{lem:typeDpyramid}.  
Having constructed a positive path from $Y$ to $V$, the see-saw property then allows us to conclude $\mu(Y) < \mu(W)$, so $\tau^k a_i$ is positive.
\end{proof}

At this point, we have shown that a stability function which is positive on all border sequences in the Auslander-Reiten quiver of a type $\mathbb{D}$ quiver is positive on every arrow, thus totally stable using Lemma \ref{lem:stabilityarrows}, completing the proof of Theorem \ref{thm:main} for quivers of type $\mathbb{D}$.

\section{Proof of main theorem in type \texorpdfstring{$\mathbb{A}$}{A}}




We assume $Q$ is a quiver of type $\mathbb{A}_n$ throughout this section and
$\mu\colon\text{rep}^*(Q)\rightarrow \mathbb{R}$ is a stability function on $Q$ which is
positive on all border sequences. To prove $\mu$ is totally stable, we follow the same
general strategy as the type $\mathbb{D}$ case. That is, we prove that $\mu$ is positive on
each arrow of AR($Q$) by going through the different regions of $\AR(Q)$ in a carefully chosen
order.
We take the convention of drawing $\text{AR}(Q)$ with $P_1$ and $I_n$ at the top
and $P_n$ and $I_1$ at the bottom. 

Just like the type $\mathbb{D}$ case, we will build on pyramids within
$\text{AR}(Q)$, except that now we have two pyramid regions. We introduce the following new
terminology for this section.

\begin{definition}
A vertex $x$ of $\text{AR}(Q)$ is in the \textit{bottom pyramid} (resp. \textit{top
pyramid}) if there exists paths in $\text{AR}(Q)$ from $P_n$ to $x$ and from $x$ to
$I_1$ (resp. paths from $P_1$ to $x$ and from $x$ to $I_n$).
\end{definition}

We note that if $Q$ is equioriented, say all arrows pointing towards vertex $n$, then the top pyramid is a single vertex and the bottom pyramid is the entirety of $\AR(Q)$.

\begin{lemma}\label{lem:typeAbase}
Each arrow in pyramid regions which has either source or target in the $\tau$-orbit of either $P_1$ or $P_n$ is positive. 
\end{lemma}
\begin{proof}
The proof of this lemma is the same as the proof of Lemma \ref{lem:bottomrow}.
\end{proof}

\begin{lemma}\label{lem:typeApyramid}
Each arrow in either of the pyramids is positive. 
\end{lemma}
\begin{proof}
The proof of this lemma is a straightforward simplification of the proof of Lemma
\ref{lem:typeDpyramid}, since we do not need special
considerations for arrows in the $\tau$-orbits of $P_1,
P_2$ and $P_3$ (i.e. the triple mesh in the type $\mathbb{D}$ pyramid).
\end{proof}

With there being both a top and bottom pyramid, we need a different technique to prove the
positivity of the remaining arrows.   
Roughly speaking, we work outwards from the pyramid walls one layer at a time.
The following key observation is immediate from Lemma \ref{lem:ladder} and the see-saw property: for a ladder as in \eqref{eq:ladder}, if each of the arrows $Z \to X_0 \to \cdots \to X_{k-1} \to X_k$ is positive, then $a$ is positive.  Using this and the corresponding dual version, we simply work outwards from the pyramids one ``layer'' at a time to complete the proof.

\begin{proposition}
Each arrow  in $\AR(Q)$ which is not in the pyramids is positive. 
\end{proposition}
\begin{proof}
We focus on arrows to the left of the bottom pyramid (i.e. arrows $a$ such that $\tau^k a$ is in the bottom pyramid for some $k < 0$) so that we can use Figure \ref{fig:perp&par} as a visual aid again.
Other arrows are covered by dual arguments.  Firstly, any such arrow with target in the pyramid wall (labeled $b_k$ in Figure \ref{fig:perp&par}) is the top rung of a dual ladder running upwards and to the right (so, drawn upside down compared to \eqref{eq:ladder}). All arrows on the right side of this ladder are in the pyramid, either in the left wall of the bottom pyramid, or the right wall of the top pyramid.  Thus our key observation above, along with Lemma \ref{lem:typeApyramid}, imply that the arrow is positive.

Next, consider an arrow $a$ such that $\tau^{-1}a$ is in the left wall of the bottom pyramid (labeled $\tau a_k$ in Figure \ref{fig:perp&par}).
Each such arrow is the top rung of a dual ladder running downwards and to the right, which is entirely contained in the bottom pyramid except for the two arrows incident to $a$ in the ladder.  Since these were shown to be positive in the previous paragraph, we get that $a$ is positive by the same reasoning as the conclusion of the previous paragraph.

Continuing outwards in this way, the region of positive arrows expands layer-by-layer and all arrows outside of the pyramids are eventually shown to be positive.
\end{proof}

This concludes the proof for the type $\mathbb{A}$ case.

\section{Proof of main theorem in type \texorpdfstring{$\mathbb{E}$}{E}}\label{sec:typeE}
Fortunately the overall strategy is the same as type $\mathbb{D}$.  The only significant difference is modifying Lemma \ref{lem:triplemeshcoker} to identify cokernels of certain morphisms in the triple meshes, which requires case-by-case consideration.
We use the numbering of type $\mathbb{E}$ vertices from Bourbaki:
\begin{equation}
\begin{tikzcd}
	&& 2 \\
	1 & 3 & 4 & 5 & 6 & 7 & 8
	\arrow[no head, from=2-1, to=2-2]
	\arrow[no head, from=2-2, to=2-3]
	\arrow[no head, from=2-3, to=2-4]
	\arrow[no head, from=2-3, to=1-3]
	\arrow[no head, from=2-4, to=2-5]
	\arrow[no head, from=2-5, to=2-6]
	\arrow[no head, from=2-6, to=2-7]
\end{tikzcd}
\end{equation}

We first establish the type $\mathbb{E}$ analogue of Lemma \ref{lem:toprows}. It is a little simpler than the type $\mathbb{D}$ case because there is only one short branch in type $\mathbb{E}$ instead of the two in type $\mathbb{D}$.

\begin{lemma}\label{lem:typeEprojarrow}
All arrows with both source and target in the $\tau$-orbits of $P_2$ and $P_4$, i.e. those through the middle of the triple meshes, are positive. 
\end{lemma}
\begin{proof}
Without loss of generality, we assume we have the orientation $2\rightarrow 4$ in our quiver.
The case where we have $2\leftarrow 4$ follows by duality. 
All such arrows are already in a border sequence except the one $P_4 \to P_2$, for which we have a short exact sequence
\begin{equation}\label{eq:ProjSESE}
0\rightarrow P_4 \rightarrow P_2\rightarrow I_2\rightarrow 0.
\end{equation}
As in the type $\mathbb{D}$ case, there is still a path of positive arrows from $P_2$ to $I_2$ running through the middle of the triple meshes, and each positive by assumption, so by the see-saw property $P_4 \to P_2$ is positive as well.
\end{proof}

There are two analogues of the pyramid region we defined for type $\mathbb{D}$.  With this in mind, in type $\mathbb{E}_n$ ($n=6, 7, 8$), we define the \emph{double pyramid region}  as the full subquiver of AR($Q$) containing those vertices $x$ such that either:
\begin{itemize}
    \item there exists a path from $P_1$ to $I_1$ through $x$, or
    \item there exists a path from $P_n$ to $I_n$ through $x$.
\end{itemize}
The left wall of the double pyramid region is the full subquiver of AR($Q$) containing those vertices $x$ such that there is either:
\begin{itemize}
    \item a unique path from $P_1$ to $x$, or
    \item a unique path from $P_n$ to $x$.
\end{itemize}

The following analogue of Lemma \ref{lem:typeDpyramid} holds by a similar proof.
\begin{lemma}\label{lem:typeEpyramid}
With the setup above, every arrow in the double pyramid region is positive.
\end{lemma}
\begin{proof}
Lemma \ref{lem:typeEprojarrow} and the hypothesis of border sequences being positive give that all arrows through the center of the triple meshes are positive.
Then the rest of the arrows in the double pyramid region are positive following essentially the same proof as Lemma \ref{lem:typeDpyramid}.
\end{proof}

From here, Notation \ref{not:bi} readily generalizes to the double pyramid, so that we may refer to arrows perpendicular and parallel to the left wall.
We need to identify the cokernels of these arrows, which are done separately in the types $\mathbb{E}_6, \mathbb{E}_7, \mathbb{E}_8$.
We label arrows in sequences of triples meshes as follows.
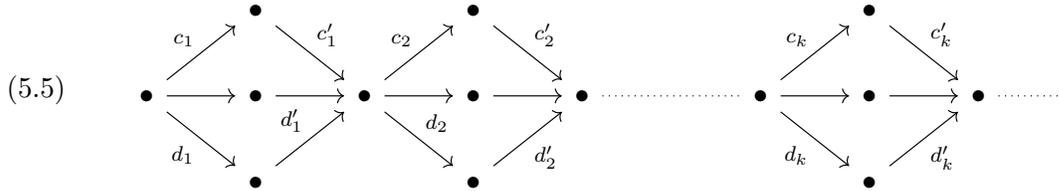
\begin{equation}\label{eq:typeEmeshes}
\begin{tikzcd}
	& \bullet && \bullet &&&& \bullet \\
	\bullet & \bullet & \bullet & \bullet & \bullet && \bullet & \bullet & \bullet & {} \\
	& \bullet && \bullet &&&& \bullet
	\arrow["{d_1}"', from=2-1, to=3-2]
	\arrow["{d_1'}", from=3-2, to=2-3]
	\arrow[from=2-1, to=2-2]
	\arrow[from=2-2, to=2-3]
	\arrow["{c_1}", from=2-1, to=1-2]
	\arrow["{c_1'}", from=1-2, to=2-3]
	\arrow["{c_2}", from=2-3, to=1-4]
	\arrow[from=2-3, to=2-4]
	\arrow["{d_2}", from=2-3, to=3-4]
	\arrow["{d_2'}"', from=3-4, to=2-5]
	\arrow[from=2-4, to=2-5]
	\arrow["{c_2'}", from=1-4, to=2-5]
	\arrow["{c_k}", from=2-7, to=1-8]
	\arrow[from=2-7, to=2-8]
	\arrow["{d_k}"', from=2-7, to=3-8]
	\arrow["{d_k'}"', from=3-8, to=2-9]
	\arrow[from=2-8, to=2-9]
	\arrow["{c_k'}", from=1-8, to=2-9]
	\arrow[dotted, no head, from=2-5, to=2-7]
	\arrow[dotted, no head, from=2-9, to=2-10]
\end{tikzcd}
\end{equation}
We assume AR($Q$) is drawn in the plane so that, with the standard identification of $Q^{\op}$ as a subquiver of AR($Q$) whose vertices are the projectives, the branch $1 - 3 - 4$ is towards the top.  In other words, either the $c_i$ or $c_i'$ arrows are $\tau$-translates of the irreducible morphism $P_4\to P_3$ or $P_3\to P_4$ (depending on the orientation of $Q$).

Now tedious but straightforward calculations as in \eqref{eq:Dcokercalc} yield the following lemma.
\begin{lemma}
For a sequence of triple meshes in AR($Q$) as shown in \eqref{eq:typeEmeshes}, we have:
\begin{equation}
    \coker c_1 = 
    \begin{cases}
    d_4' & \text{for }Q \text{ of type }\mathbb{E}_6\\
    c_7' & \text{for }Q \text{ of type }\mathbb{E}_7\\
    c_{13}' & \text{for }Q \text{ of type }\mathbb{E}_8\\
    \end{cases}
\qquad \text{and} \quad
    \coker d_1 = 
    \begin{cases}
    c_4' & \text{for }Q \text{ of type }\mathbb{E}_6\\
    d_6' & \text{for }Q \text{ of type }\mathbb{E}_7\\
    d_{11}' & \text{for }Q \text{ of type }\mathbb{E}_8\\
    \end{cases}
\end{equation}
\end{lemma}

With this, we can complete the proof by considering arrows perpendicular to, then parallel to the double pyramid region using essentially the same proof as in type $\mathbb{D}$.

\begin{lemma}\label{lem:pyramidperpE}
Each arrow perpendicular to the left pyramid wall is positive.
\end{lemma}
\begin{proof}
The analogue of Lemma \ref{lem:pyramidperp} holds by a similar proof, except that in type $\mathbb{E}_6$ the cokernel of each such arrow is located on the opposite side of the triple mesh, still within the double pyramid and still an injective representation. Since every arrow in the double pyramid is positive by Lemma \ref{lem:typeEpyramid}, we still have a positive path from the middle term of each such short exact sequence analogous to those in \eqref{eq:pyramidperp} to the corresponding injective cokernel.  This handles the first layer of arrows perpendicular to the pyramid wall, those whose targets are in the wall.
The remaining perpendicular arrows are $\tau$-translates of these, and the same proof as Lemma \ref{lem:typeDperp} works exactly as written there.
\end{proof}

Finally, the analogue of Lemma \ref{lem:typeDparallel}, that all arrows parallel to the left wall of the double pyramid are positive, holds by essentially the same proof since it only uses ladder regions.  Thus we have completed the proof in the type $\mathbb{E}$ cases.

\bibliographystyle{alpha}
\bibliography{totalstability}

\end{document}